\documentclass[10pt,twoside, a4paper, english, reqno]{amsart}
\usepackage{graphicx, amsmath, varioref, amscd, amssymb,color, bm, stmaryrd, amsthm, epsfig, color, epic}
\usepackage{mathtools}
\usepackage{fancybox}
\usepackage[pdftex, colorlinks=true]{hyperref}
\usepackage{enumerate}
\usepackage{upref}
\usepackage{pdfsync}

\usepackage{eso-pic}
\usepackage{color}
\usepackage{type1cm}
\makeatletter
  \AddToShipoutPicture*{%
    \setlength{\@tempdimb}{.12\paperwidth}%
    \setlength{\@tempdimc}{.5\paperheight}%
    \setlength{\unitlength}{1pt}%
    \put(\strip@pt\@tempdimb,\strip@pt\@tempdimc){%
      \makebox(0,0){\rotatebox{90}{\textcolor[gray]{0.75}{\fontsize{2cm}{2cm}}}}
    }
}

\makeatother

\numberwithin{equation}{section}
\allowdisplaybreaks

\newtheorem{theorem}{Theorem}[section]
\newtheorem{lemma}[theorem]{Lemma}

\theoremstyle{definition}
\newtheorem{definition}[theorem]{Definition}

\theoremstyle{remark}
\newtheorem{remark}[theorem]{Remark}
\newtheorem{notation}[theorem]{Notation}
\newcommand{\lint}[1]{\int_{0}^T\!\!\int_{\dom} #1 \, dxdt} 

\newcommand{\wu}{\widehat{u}}
\newcommand{\supp}{\operatorname{supp}}

\newcommand{\Div}{\operatorname{div}}

\newcommand{\Grad}{\nabla}

\newcommand{\vu}{\vc{u}}

\newcommand{\ihat}{\hat{\imath}}
\newcommand{\jhat}{\hat{\jmath}}

\newcommand{\vc}[1]{{\bm{#1}}}

\newcommand{\weak}{\rightharpoonup}

\newcommand{\weakp}[2]{\overline{#1\strut}\;\overline{#2\strut}}
\newcommand{\weakl}[1]{{\overline{#1}}}

\newcommand{\n}{n} 
\newcommand{\nmax}{\overline{\n}_{\infty}} 
\newcommand{\W}{W} 
\newcommand{\dom}{\Omega}
\newcommand{\G}{\vc{G}}

\newcommand{\cell}{\mathcal{C}}

\newcommand{\Dt}{\Delta t}
\newcommand{\R}{\mathbb{R}}

\newcommand{\Dom}{(0,T)\times\Omega}

\newcommand{\e}{\varepsilon}

\newcommand{\eps}{\epsilon}

\begin{document}

\title[On a tumor growth model] {A convergent explicit finite difference scheme for  a mechanical model for tumor growth}

\author[Trivisa]{Konstantina Trivisa}
\address[Trivisa]{\newline
Department of Mathematics \\ University of Maryland \\ College Park, MD 20742-4015, USA.}
\email[]{\href{http://www.math.umd.edu}{trivisa@math.umd.edu}}
\urladdr{\href{http://www.math.umd.edu/~trivisa}{math.umd.edu/\~{}trivisa}}

\author[Weber]{Franziska Weber}
\address[Weber]{\newline
Departement of  Mathematics\\
University of Oslo\\
0316 Oslo, Norway.}
\email[]{\href{franziska.weber@cma.uio.no}{franziska.weber@cma.uio.no}}

%
%

\date{\today}



\thanks{}

\maketitle

\begin{abstract}
Mechanical models for tumor growth have been used extensively in recent years for the analysis of medical observations and for the prediction of cancer evolution based on imaging analysis.
This work deals with the numerical approximation of a mechanical model for tumor growth and the analysis of  its dynamics.
The system under investigation is given by a multi-phase flow model: The densities of the different cells are governed by a  transport equation for the evolution of tumor cells, whereas the velocity field  is given by a Brinkman regularization of the classical Darcy's law. 
An efficient finite difference scheme is proposed and shown to converge to a weak solution of the system. Our approach relies on convergence and compactness arguments in the spirit of Lions \cite{Lions-1998}.
\end{abstract}

\section{Introduction}\label{S1}

\subsection{Motivation}
Mechanical models for tumor growth are used extensively in recent years for  the prediction of cancer evolution based on imaging analysis. Such models are based on the assumption that the growth of the tumor is mainly limited by the competition for space.
Mathematical modeling, analysis and numerical simulations together with experimental and clinical observations are essential components in the effort to enhance our understanding of the cancer development. 
The goal of this article is to make a further step  in the investigation of such models by  presenting  a convergent explicit finite difference scheme for the numerical approximation of a Hele-Shaw-type model for tumor growth and  by providing its detailed mathematical analysis.  Even though the main focus in the present work is on the investigation of the evolution of the proliferating cells, it provides a mathematical framework that can potentially accommodate more complex systems that account for the presence of nutrient and drug application. This will be the subject of future investigation~\cite{TW-prep2015}.

\subsection{Governing equations}\label{1.2}
In the present context the tissue is considered as a multi-phase fluid and the ability of the tumor to expand into a host tissue is then primarily driven by the cell division rate which depends on the
local cell density and the mechanical pressure in the tumor.
\subsubsection{Transport equations for the evolution of the cell densities}
The dynamics of the cell population density $\n(t, x)$ under pressure forces and cell multiplication is described by a transport equation
\begin{equation}
{\partial_t}\n - \Div(\n  {\vc{u}}) = \n \vc{G}(p), \quad x \in \dom, \,\, t \ge 0 \label{cell-n}
\end{equation}
where $\n$ represents the number   density of tumor cells, $\vu$ the velocity field and  $p$ the pressure of the {\em tumor}. $\dom$ is a bounded domain in $\R^d$, $d=2,3$. The pressure law is given by 
\begin{equation}
p (n) = a n^{\gamma}, \label{pressure-p}
\end{equation}
where $\gamma\geq 2$.
Following \cite{ByrneDrasdo-2009, Ranftetal-2010},  we assume that growth is directly related to the pressure through a function $\vc{G}(\cdot)$ which satisfies
\begin{equation}
\vc{G}\in C^1(\R), \quad \vc{G}'(\cdot) \le -\beta< 0, \quad \vc{G}(P_M) =0 \quad \mbox{for some} \,\, P_M >0. \label{G-condition}
\end{equation}
The pressure $P_M$ is usually called {\em homeostatic pressure.}
Here, and in what follows, for simplicity we let
\begin{equation}
\vc{G}(p) = \alpha - \beta p^\theta, \label{G-example}
\end{equation}
for some $\alpha,\beta,\theta>0$.
\smallskip

\subsubsection{The tumor tissue as a porous medium}
The continuous motion of cells within the  tumor region, typically due to proliferation, 
 is represented by the velocity field $\vc{u}:=\Grad\W$ given by an alternative to Darcy's equation known as {\em Brinkman's equation}

\begin{equation}
 p  = \W - \mu \Delta \W \label{brinkman-u}
\end{equation}
where  $\mu$  is a positive constant describing the viscous like properties of tumor cells
and $p$ is the pressure given by \eqref{pressure-p}.


Relation \eqref{brinkman-u} consists of  two terms. The first term is the usual Darcy's law, which in the  present setting describes the tendency of cells to move down pressure gradients and results from  the friction of the tumor cells with the extracellular matrix.  The second term, on the other  hand, is a  dissipative force density (analogous to the Laplacian term that appears in the Navier-Stokes equation) and results from the internal cell friction due to cell volume changes.
A second interpretation of   relation \eqref{brinkman-u} is the tumor tissue can be viewed as ``fluid like." In other words,   the tumor cells flow through the fixed extracellular matrix like a flow through a porous medium, obeying Brinkman's law.

The resulting model, governed by the transport equation \eqref{cell-n}  for the population density of cells,  the elliptic equation \eqref{brinkman-u} for the velocity field and a state equation for the pressure law \eqref{pressure-p}, now reads
\begin{equation}
\begin{cases} \label{HeleShaw}
&\partial_t \n - \Div(\n  \Grad \W) =  \alpha  \n- \beta \n^{\gamma\theta +1}, \quad x \in \dom, \,\, t \ge 0 \\
&    -\mu \Delta \W + \W = a  \n^{\gamma}.
\end{cases}
\end{equation}
We complete the system \eqref{HeleShaw} with a family of initial data $n_0$ satisfying (for some constant $C$)
\begin{equation}
n_0 \ge 0, \quad p(n_0) \le P_M, \quad \|n_0\|_{L^1(\R^d)} \le C. \label{IC}
\end{equation}

The objective of this work is to establish the global existence of  weak solutions to the nonlinear model for tumor growth \eqref{HeleShaw} by  designing an efficient numerical scheme for its approximation and by showing  that this scheme converges when the mesh is refined. The main ingredients of our approach and contribution to the existing theory  include:

\begin{enumerate}[\quad$\bullet$]
\item The introduction of a suitable notion of solutions to the nonlinear system \eqref{HeleShaw} consisting of the transport equation \eqref{cell-n} and the  Brinkman regularization \eqref{brinkman-u}. 

\item The construction of an approximating procedure which relies on an artificial vanishing viscosity approximation and the establishment of the suitable compactness in order to pass into the limit and to conclude convergence to the original system (cf. Section \ref{S3}, Lemma 3.7).

\item  The design of an efficient numerical scheme for the numerical  approximation of the nonlinear system \eqref{cell-n}-\eqref{brinkman-u}. 
\item The proof of the convergence of the numerical scheme. In the center of the analysis lies the proof of the strong convergence of the cell densities. This is achieved  by establishing the weak continuity of the {\em effective viscous pressure} in the spirit of Lions \cite{Lions-1998} (cf. Section \ref{S4}, Lemma 4.8).

\item The design of  numerical experiments in order to establish  that
the finite difference scheme is effective in computing approximate solutions
to the nonlinear system \eqref{HeleShaw} (cf. Section 5).
\end{enumerate}

For relevant results on the analysis and the  numerical approximation of a two-phase flow model in porous media we refer the reader to \cite{CocliteMishraRisebroWeber-2014}. Related results on the numerical approximation of compressible fluids employing the weak compactness tools developed by of Lions~\cite{Lions-1998} in the discrete setting have been established by Karper \emph{et al.} \cite{Karper2013NumMath,KarlsenKarper2010,KarlsenKarper2011,KarlsenKarper2012} and Gallou\"{e}t \emph{et al.} \cite{EymardGallouet2010}.

Relevant work on the mathematical analysis of mechanical models of Hele-Shaw-type have been presented by
Perthame \emph{et al.} \cite{PerthameQTV-2014,  PerthameQV-2014,  PerthameTV-2014, PerthameV-2014}. The analysis in \cite{PerthameTV-2014} establishes the existence of traveling wave solutions of the Hele-Shaw model of tumor growth with nutrient 
and presents numerical observations in two space dimensions. The present article is according to our knowledge the first article presenting rigorous analytical results on the global existence of general weak solutions  to  Hele-Shaw-type systems.  

A different approach yielding results on the global existence of weak solutions to  a nonlinear model for tumor growth in a general moving domain $\Omega_t \subset \R^3$ without any symmetry assumption and for finite large initial data is presented in  \cite{DT-MixedModelJMFM-2014, DT-VariableDensity-2014, DT-DrugApplication-2014}.
But in contrast to the present nonlinear system, the transport equation for the evolution of cancerous cells in   \cite{DT-MixedModelJMFM-2014, DT-DrugApplication-2014} has a source term which is linear with respect to cell density.

Relevant results on nonlinear models for tumor growth governed by the Darcy's law for the evolution of the velocity field are presented 
by Zhao \cite{Zhao-2010} based on the farmework introduced  by Friedman {\em et al.}  \cite{Friedman-2004, ChenFriedman-2013}.
\subsection{Outline}
The paper is organized as follows: Section \ref{S1} presents the motivation, modeling  and   introduces the necessary preliminary material. Section \ref{S2}  provides a weak formulation of the problem and states the main result. Section \ref{S3}  is devoted to the global existence of solutions via a vanishing viscosity approximation. 
In Section \ref{S4}  we present an efficient finite difference scheme for the approximation of the weak solution to system \eqref{HeleShaw} on rectangular domains and
 Section \ref{S5} is devoted to numerical experiments.  A  discretized Aubin-Lions lemma  and some technical lemmas are presented in Appendices A and B respectively.

\section{Weak formulation and main results}\label{S2}
\begin{notation}
For $\varphi:\Dom\rightarrow \R$, $\vc{\varphi}:\Dom\rightarrow\R^d$, we will denote by $\Grad \varphi:=\Grad_x\varphi=(\partial_{x_1}\varphi,\dots,\partial_{x_d}\varphi)$ and $\Div\vc{\varphi}:=\Div_x\vc{\varphi}=\sum_{i=1}^d \partial_{x_i}\vc{\varphi}^{(i)}$ the gradient and divergence in the spatial direction in $\dom$.
\end{notation}
\subsection{Weak solutions}
\begin{definition}\label{D2.1}
Let $\dom$ a bounded domain in $\R^d$, $d=2,3$, which is either rectangular or has a smooth boundary $\partial\dom$ and $T>0$ a finite time horizon. We say that $(\n,\W, p)$ is a weak solution of problem \eqref{cell-n}-\eqref{brinkman-u} supplemented with initial data $(n_0,  \W_0, p_0)$
satisfying (\ref{IC}) provided that the following hold:
\vspace{0.1in}

$\bullet$ $(\n,\W, p) \ge 0$ represents a weak solution of \eqref{cell-n}-\eqref{brinkman-u} on $\Dom$, i.e., for any test function $\varphi \in C^{\infty}_c ([0,T]\times\R^d), T>0$, 
the  following integral relations hold
\smallskip

\begin{equation}
{\int_{\R^d}\!  \n \varphi(\tau,\cdot) \, dx  -\! \int_{\R^d}\!  \n_0 \varphi(0,\cdot)dx =}
{ \int_0^{\tau} \!\!\int_{\R^d} \!\!\left( n \partial_t \varphi - n \Grad \W \cdot \Grad \varphi + \n\vc{G}(p)  \varphi(t, \cdot) \right) dx dt}.
 \label{weak-n}
\end{equation}
%
In particular, 
$$n \in L^p(\Dom),  \,\, \mbox{for all}\,\, p \ge 1. $$ 

We remark that in  the weak formulation, it is convenient that the equations \eqref{cell-n} hold in the whole space $\mathbb{R}^d$ provided that the densities $n$ are extended to be zero outside the tumor domain.
\smallskip

$\bullet$ Brinkman's equation \eqref{brinkman-u} holds in the sense of distributions, i.e., for any test function 
${\varphi} \in C^{\infty}_c(\R^d)$ satisfying 
$${\varphi}|_{\partial \dom} = 0\,\, \mbox{for any}\,\, t \in [0,T],$$ 
the following integral relation holds for a.e. $t\in [0,T]$,
\begin{equation}
\int_{\dom} a n^{\gamma} {\varphi} \, dx=
 \int_{\dom} \Big( \mu \Grad \W  \cdot \Grad {\varphi} +\W {\varphi}   \Big) dx.
\label{w-pressure2}
\end{equation}
and $p=\n^\gamma$ almost everywhere. All quantities in \eqref{w-pressure2} are required to be integrable, and in particular, $\W \in L^{\infty}([0,T];H^{2}(\dom)).$
\end{definition}

The main result of the article now follows. 

\begin{theorem}\label{T2.2}
Let $\dom\subset\R^d$ be a bounded domain with smooth boundary $\partial\dom$, $0<T<\infty$. Assume that the initial data $\n_0\in L^{\infty}(\dom)$ with $0\leq\n_0\leq \n_\infty:=P_M^{1/\gamma}$ and that $\G(\cdot)$ is of the form \eqref{G-example}. Then the problem \eqref{cell-n}-\eqref{brinkman-u},   admits a weak solution in the sense specified in Definition \ref{D2.1}. 
\end{theorem}
The following two remarks are now in order.
\begin{remark}
In Section \ref{S3}, such a solution is  obtained as the limit of the vanishing viscosity approximations   $(\n_\e,\W_\e,p_\e)$ of  \eqref{HeleShaw-appr}  to \eqref{HeleShaw} as $\e\rightarrow 0$. 
\end{remark}
\begin{remark}
In Section \ref{S4}, such a solution  is obtained in the case of a rectangular domain,  as the limit of the sequence of approximations $(\n_h,\W_h,p_h)$ computed by the numerical scheme \eqref{eq:discv} -- \eqref{eq:fluxes} as $h\rightarrow 0$. 
\end{remark}

%
%
\section{Global existence via vanishing viscosity}\label{S3}
In this section we prove Theorem \ref{T2.2} by constructing an approximating scheme which relies on the addition of an artificial vanishing viscosity approximation
\begin{equation}
\begin{cases} \label{HeleShaw-appr}
&\partial_t \n_\e - \Div(\n_\e  \Grad\W_\e) =  \alpha  \n_\e- \beta \n^{\gamma +1}_\e + \e \Delta \n_\e, \quad x \in \dom, \,\, t \ge 0 \\
&    \mu \Delta \W_\e - \W_\e = a \n_\e^{\gamma},\\
&\n_\e(0,\cdot)=\n_0^\e,
\end{cases}
\end{equation}
where $\n_0^\e$ is a smoothnened version of $\n_0$, that is $\n_0^\e=\n_0\ast\varphi_\e$ for a smooth function $\varphi_\e$ with compact support, and a bounded domain $\dom\in\R^d$ with smooth boundary or alternatively the $d$-dimensional torus $\mathbb{T}^d$, and we establish its convergence to the nonlinear system  \eqref{HeleShaw}  at the continuous  level. For simplicity, we assume $a=1$ and homogeneous Neumann boundary conditions for $\n_\e$ and $\W_\e$ (if the domain is a torus $\mathbb{T}^d$ we can also use periodic boundary conditions).
 \begin{theorem}
 For every $\e > 0$, the parabolic-elliptic system \eqref{HeleShaw-appr} admits a unique smooth solution $(\n_\e, \W_\e, p_\e)$.  
 \end{theorem}
\begin{proof}
The proof of this result relies on classical arguments  (cf.\ Ladyzhenskaya \cite{Ladyzhenskaya-1969}), namely by employing the Contraction Mapping Principle and the regularity of the initial data one can show the existence of a unique solution  $(\n_\e, \W_\e, p_\e)$  defined for a small time $T>0.$ Then one derives apriori estimates  establishing that the solution does not blow up and in fact is defined for every time. Finally, a bootstrap argument yields the smoothness of the solution.
\end{proof}
The remaining part of this section aims to establish the necessary compactness of the approximate sequence of solutions  $(\n_\e, \W_\e, p_\e).$
\subsection{A priori estimates}
We start by proving that $\n_\e$ are uniformly bounded independent of $\e>0$ and nonnegative:
\begin{lemma}\label{lem:linfne}
If $0\leq \n_\e(0,\cdot)\leq \n_{\infty}:=P_M^{1/\gamma}<\infty$  uniformly in $\e>0$, then for any $t>0$, the functions $\n_\e(t,\cdot)$ are uniformly {\normalfont(}in $\e>0${\normalfont)} bounded and nonnegative, specifically,
\begin{equation*}
0\leq \min_{(t,x)}\n_\e(t,x)\leq \max_{(t,x)} \n_{\e}(t,x)\leq \n_{\infty}.
\end{equation*}
\end{lemma}
\begin{proof}
First we notice that if $\W_\e$ has a maximum at a point $x_0$, then $\Delta\W_\e(\cdot,x_0)\leq 0$ and therefore $\W_\e=p_\e+\mu\Delta\W_\e\leq p_\e$. Similarly, if it has a minimum at a point $x_0$, it will satisfy $\Delta\W_\e(\cdot,x_0)\geq 0$ and therefore $\W_\e\geq p_\e$.
If $\W_\e$ attains a strict maximum on the boundary, i.e., there is a point $x_0\in \partial\dom$ such that $\W_\e(x_0)>\W_\e(x)$ for any other $x\in\dom$, we apply Hopf's Lemma, e.g. \cite[p. 347]{Evans2010}, to the function $v:=\W_\e-\max_{(t,x)} p_\e(t,x)$ which satisfies
\begin{equation*}
-\mu\Delta v+v=p_\e-\max_{(t,x)} p_\e(t,x)\leq 0,
\end{equation*}
 which has a strict maximum at the point $x_0$. If $v(x_0)\leq 0$, then $\W_\e\leq \W_\e(x_0)\leq \max_{(t,x)} p_\e(t,x)$ and otherwise Hopf lemma gives $\Grad \W_\e(x_0)\cdot \nu=\Grad v(x_0)\cdot \nu>0$ where we have denoted the boundary normal $\nu$, this contradicts the homogeneous boundary conditions. In a similar way we show that $\W_\e\geq \min_{(t,x)} p_\e(t,x)$ (applying Hopf's lemma to $-\W_\e$ and hence 
\begin{equation}\label{eq:maxW}
\min_{(t,x)} p_\e(t,x)\leq \W_\e \leq \max _{(t,x)} p_\e(t,x).
\end{equation}
We rewrite the evolution equation for $\n_\e$ using the equation for the potential $\W_\e$,
\begin{equation}\label{eq:noncons}
\partial_t \n_{\e} - \Grad\W_\e\cdot \Grad\n_\e= \n_\e \G(p_\e)+\frac{1}{\mu}\n_\e (p_\e - \W_\e)+\e\Delta\n_\e.
\end{equation} 
Now assume $(t_0,x_0)$ is a point, where $\n_\e(t_0,x_0)\geq \n_\infty$  reaches its maximum (and therefore also $p_\e(t_0,x_0) \geq P_M$ reaches a maximum). Then 
$\Grad \n_\e(t_0,x_0)=0$ and $\Delta \n_\e(t_0,x_0)\leq 0$. Hence
\begin{equation*}
\partial_t \n_{\e}(t_0,x_0)\leq \n_\e \G(p_\e)+\frac{1}{\mu}\n_\e (p_\e - \W_\e).
\end{equation*} 
By \eqref{eq:maxW}, the second term on the right hand side is nonpositive  and since $\G(p_\e(t_0,x_0))\leq 0$ for $p_\e\geq P_M$, we get
\begin{equation*}
\partial_t \n_{\e}(t_0,x_0)\leq 0.
\end{equation*} 
Hence $\n_\e$ will decrease and if initially $\n_0\leq \n_\infty$, this implies that $\n_\e(t,\cdot)\leq \n_\infty$ for any later time $t\geq 0$. To show the nonnegativity of $\n_\e$, we integrate the evolution equation for $\n_\e$, 
\begin{equation*}
\frac{d}{dt}\int_{\dom}\n_\e dx=\int_{\dom}\n_\e\G(p_\e) dx.
\end{equation*}
On the other hand, multiplying the same equation by a regularized version of the sign function, integrating and then passing to the limit in the approximation, we have
\begin{equation*}
\frac{d}{dt}\int_{\dom}|\n_\e| dx\leq\int_{\dom}|\n_\e|\G(p_\e) dx,
\end{equation*}
Subtracting the two equations from one another, and using that $|\n_\e|-\n_\e\geq 0$,
\begin{equation*}
\begin{split}
\frac{d}{dt}\int_{\dom}\big||\n_\e|-\n_\e\big| dx&\leq\int_{\dom}\big||\n_\e|-\n_\e\big|\G(p_\e) dx,\\
&\leq \max_{s\in[0,P_M]}|\G(s)|\int_{\dom}\big||\n_\e|-\n_\e\big| dx.
\end{split}
\end{equation*}
Now using Gr\"{o}nwall's inequality and that $|\n_0|-\n_0\equiv 0$ by assumption, we obtain
$$
\int_{\dom}\big||\n_\e|-\n_\e\big|(t) dx=0
$$
and thus that $\n_\e(t,x)\geq 0$ almost everywhere.
%
\end{proof}
%
Next we prove a simple lemma on the regularity of $\W_\e$.
\begin{lemma}\label{lem:We}
We have that 
$$
\W_\e\subset L^\infty([0,T];W^{2,q}(\dom)),
$$ for any $q\in[1,\infty)$ uniformly in $\e>0$ and 
$$
\W_\e, \Delta \W_\e\subset L^\infty(\Dom)),
$$
uniformly in $\e>0$ as well.
\end{lemma}
\begin{proof}
We square the equation for $\W_\e$ and integrate it over the spatial domain and then use integration by parts,
\begin{equation*}
\begin{split}
\int_{\dom} |p_\e|^2dx&=\int_{\dom} | \W_\e|^2-2\mu\W_\e\Delta\W_\e+\mu^2 |\Delta\W_\e|^2dx\\
&=\int_{\dom} | \W_\e|^2+2\mu|\Grad\W_\e|^2+\mu^2 |\Grad^2\W_\e|^2dx.
\end{split}
\end{equation*}
By the previous Lemma \ref{lem:linfne}, we have that $p_\e$ is uniformly bounded in $\e>0$ and therefore that the left hand side of the above equation is bounded and that $\W_\e\in L^\infty([0,T];H^2(\dom))$. Using a Calderon-Zygmund inequality (e.g. \cite[Thm. 9.11.]{gilbargtrudinger}), we obtain $\W_\e\in L^{\infty}([0,T]; W^{2,q}(\dom))$ for all $q\in [1,\infty)$. By the Sobolev embedding theorem, this implies that in particular $\Grad \W_\e\in L^{\infty}(\Dom)$. The second claim follows from \eqref{eq:maxW} and the uniform bound on the pressure proved in Lemma \ref{lem:linfne}.
\end{proof}
\subsection{Entropy inequalities for $\n_\e$}
To prove strong convergence of the approximating sequence $\{(\n_\e,\W_\e,p_\e)\}_{\e>0}$, it will be useful to derive entropy inequalities for $\n_\e$. To this end, the following lemma will be useful:
\begin{lemma}\label{lem:fn}
Let $f:\R\rightarrow \R$ be a smooth convex, nonnegative function and denote $f_\e:=f(\n_\e)$. Then $f_\e$ satisfies the following identity
\begin{multline}\label{eq:entie}
\partial_t f_\e-\Div(f_\e\Grad \W_\e)-\e\Delta f(\n_\e)\\
= (f'(\n_\e)\n_\e-f_\e)\Delta \W_\e+f'(\n_\e)\n_\e \G(p_\e)-\e f''(\n_\e) |\Grad\n_\e|^2
\end{multline}

where 
\begin{equation}\label{eq:dritte}
\begin{split}
\e\lint{f''(\n_\e) |\Grad\n_\e|^2}\leq C,
\end{split}
\end{equation}
with $C>0$ a constant independent of $\e>0$. In particular, this implies that $\partial_t f_\e=g_\e+k_\e$ with $g_\e\in L^1([0,T]\times \dom)$ and $k_\e\in L^1([0,T]; W^{-1,2}(\dom))$.
\end{lemma}
\begin{proof}
The identity \eqref{eq:entie} follows after multiplying the evolution equation for $\n_\e$, \eqref{eq:noncons}, by $f'(\n_\e)$ and using chain rule. Integrating the inequality in space and time, we obtain
\begin{multline*}
\int_{\dom} f_\e(T)\, dx+\e\lint{f''(\n_\e)|\Grad \n_\e|^2}\\ =\int_{\dom}f_\e(0)\, dx+\lint{(f'(\n_\e)\n_\e-f_\e)\Delta\W_\e+f'(\n_\e)\n_\e\G(p_\e)}
\end{multline*}
The right hand side is bounded by the assumptions on the initial data and the $L^\infty$-bounds proved in Lemmas \ref{lem:linfne} and \ref{lem:We}. This implies \eqref{eq:dritte}.
Therefore the right hand side of \eqref{eq:entie} is contained in $L^1(\Dom)$. Using \eqref{eq:dritte} for the third term on the left hand side, we conclude that it is contained in $L^1([0,T]; H^{-1}(\dom))$. The second term on the left hand side is contained in $L^\infty([0,T];W^{-1,2}(\dom))$. Hence $\partial_t f_\e=g_\e+k_\e$ with $g_\e\in L^1([0,T]\times \dom)$ and $k_\e\in L^1([0,T]; W^{-1,2}(\dom))$ and in particular, $\partial_t f_\e\in L^1([0,T];W^{-1,1^*}(\dom))$ by the Sobolev embedding ($1^*=d/(d-1)$).
\end{proof}
\begin{remark}
The preceeding lemma implies that the time derivative of the approximation of the pressure $\partial_t p_\e=\partial_t|\n_h|^\gamma=g_\e+k_\e$ where $g_\e$ is uniformly bounded in $L^1([0,T]\times \dom)$ and $k_\e$ in $L^1([0,T]; H^{-1}(\dom))$. Hence $\partial_t \W_\e=U_\e+V_\e$ where $U_\e\in L^1([0,T];H^1(\dom))$ solves $-\mu\Delta U_\e+U_\e=k_\e$ and $V_\e\in L^1([0,T];W^{1,r}(\dom))$, $1\leq r<1^*$ solves $-\mu\Delta V_\e+V_\e=g_\e$ (see \cite[Thm. 6.1]{Benilan1995} for a proof of the second statement). Hence $\partial_t\W_\e\in L^1([0,T];W^{1,r}(\dom))$ for any $1\leq r<1^*$.
\end{remark}

\subsection{Passing to the limit $\e\rightarrow 0$}
The estimates of the previous (sub)sections allow us to pass to the limit $\e\rightarrow 0$ in a subsequence, still denoted $\e$, and conclude the existence of limit functions
\begin{align*}
\n_\e\weak \n\geq 0,&\quad \mathrm{in}\, L^q([0,T]\times \dom), \, 1\leq q<\infty,\\
p_\e\weak \weakl{p}\geq 0,&\quad \mathrm{in}\, L^q([0,T]\times \dom), \, 1\leq q<\infty,
\end{align*}
where $p_\e:=\n_\e^\gamma$ and $0\leq \n,\weakl{p}\in L^{\infty}([0,T]\times \dom)$. Using Aubin-Lions' lemma  for $\W_\e$ and $\Grad \W_\e$, we obtain strong convergence of a subsequence in $L^q([0,T]\times\dom)$ for any $q\in [0,\infty)$ to limit functions $\W, \Grad\W\in L^q([0,T]\times \dom)$. Moreover, from the estimates in Lemma \ref{lem:We} we obtain that $\W\in L^{\infty}([0,T]\times\dom)\cap L^{\infty}([0,T];W^{2,q}(\dom))$.
Hence we have that $(\n,\W,\overline{p})$ satisfy for any $\varphi,\psi \in C^1_0([0,T)\times\dom)$,
\begin{align}\label{eq:weaklim}
\begin{split}
\lint{\n\varphi_t-\n\Grad\W\cdot \Grad\varphi}+\int_{\dom}\!\n_0\,\varphi(0,x)dx&=-\lint{\overline{\n\G({p})}\varphi}\\
\lint{\W\psi+\mu\Grad\W\cdot\Grad\psi}&=\lint{\overline{p}\,\psi}
\end{split}
\end{align}
where $\overline{\n\G({p})}$ is the weak limit of $\n_\e\G(p_\e)$. %
To conclude that the limit $(\n,\W,p)$ is a weak solution of \eqref{HeleShaw}, we need to show that $\n_\e$ converges strongly and therefore in the limit $\overline{p}=p:=\n^\gamma$ and $\overline{\n\G({p})}=\n\G(p)$. 
For this purpose, we combine a compensated compactness property (Lemma \ref{lem:effve}) with a monotonicity argument.
We will also make use of the following lemma which was proved in a more general version in \cite{dipernalions1989,Novotny-Stras-2004}:
\begin{lemma}\label{lem:renormalized}
Let $\n,f\in L^{\infty}([0,T]\times\dom)$ and $\vc{u}\in L^{\infty}([0,T];H^1(\dom))$ with $\Div \vc{u}\in L^{\infty}([0,T]\times\dom)$ satisfy 
\begin{equation}\label{eq:transport1}
\n_t-\Div(\vc{u}\n)= f,
\end{equation}
in the sense of distributions. Then for all continuously differentiable functions $b\in C^1(\R)$,
\begin{equation}\label{eq:renormalized}
b(\n)_t-\Div(\vc{u}b(\n))= b'(\n) f+ [b'(\n)\n-b(\n)]\Div\vc{u},
\end{equation}
in the sense of distributions.
\end{lemma}
\begin{proof}
We let $0\leq \psi\in C^\infty_0(\R^{d+1})$ be a smooth, radially symmetric mollifier, i.e. $\psi(x)=\psi(-x)$ and $\int_{\R^{d+1}} \psi(x) dx$, with $\supp(\psi)\subset B_1(0)$ and denote for $\delta>0$, $\psi_\delta(x):=\delta^{-(d+1)}\psi(x/\delta)$. Then we choose as a test function in \eqref{eq:transport1} $\psi_\delta(s,y)\varphi(t+s,x+y)$, with $\varphi$ is compactly supported in $(\delta,T-\delta)\times\dom^\delta$ where $\dom^\delta$ includes all the points $x$ in $\dom$ which have distance $d(x,\partial\dom)>\delta$ and do a change of variables:
\begin{multline*}
\lint{\n(t-s,x-y)\psi_\delta(s,y)\partial_t\varphi(t,x)-\n(t-s,x-y)\vc{u}(t,x)\psi_\delta(s,y)\cdot\Grad\varphi(t,x)}\\=-\lint{f(t-s,x-y)\psi_\delta(s,y)\varphi(t,x)}.
\end{multline*}
Integrating in $(s,y)$, this becomes
\begin{multline*}
\lint{(\n\ast\psi_\delta)(t,x)\partial_t\varphi(t,x)-(\n\vc{u})\ast \psi_\delta (t,x)\cdot\Grad\varphi(t,x)}\\=-\lint{(f\ast \psi_\delta)(t,x)\varphi(t,x)}.
\end{multline*}
We define $\n_\delta:=\n\ast\psi_\delta$ and $f_\delta:=f\ast\psi_\delta$ and choose as a test function $\varphi:=b'(\n_\delta)\phi$ for a smooth $\phi$ compactly supported in $(\delta,T-\delta)\times\dom^\delta$ (which is possible since $\n_\delta$ is smooth and bounded thanks to the convolution.). Then we can rewrite the last identity using chain rule as
 \begin{multline*}
\lint{b(\n_\delta)\partial_t\phi-b(\n_\delta)\vc{u}\cdot\Grad\phi} \\=-\lint{\left(b'(\n_\delta)f_\delta+[b'(\n_\delta)\n_\delta-b(\n_\delta)]\Div\vc{u} +b'(\n_\delta) r_\delta \right)\phi}.
\end{multline*}
where $r_\delta:=\Div((\n\vc{u})\ast\psi_\delta)-\Div(\n_\delta \vc{u})$. By \cite[Lemma 2.3]{Lions-1998-1}, we have that $r_\delta\rightarrow 0$ in $L^2_{\mathrm{loc}}(\Dom)$ and thanks to the properties of the convolution that $b(\n_\delta)\rightarrow b(\n)$ almost everywhere as well as $f_\delta\rightarrow f$ a.e. when $\delta\rightarrow 0$.
Thus we obtain that in the limit $\delta\rightarrow 0$, $\n$ satisfies
\begin{equation*}
\lint{b(\n)\partial_t\phi-b(\n)\vc{u}\cdot\Grad\phi} =-\lint{\left(b'(\n)f+[b'(\n)\n-b(\n)]\Div\vc{u}  \right)\phi}.
\end{equation*}
which is exactly \eqref{eq:renormalized} in the sense of distributions.
\end{proof}
Applying Lemma \ref{lem:renormalized} for the weak limit $\n$ in \eqref{eq:weaklim} with $b(\n)=\n^2$, we obtain that $\n$ satisfies
\begin{equation}\label{eq:faen1}
\lint{\n^2\varphi_t-\n^2\Grad\W\cdot \Grad\varphi}=-\lint{(2n\overline{\n\G({p})}+n^2\Delta\W)\varphi}
\end{equation}
for any test functions $\varphi \in C^1_0((0,T)\times\dom)$.
On the other hand, from \eqref{eq:entie} for $b(\n)=\n^2$ we obtain after integrating in space and time
\begin{equation*}
\int_{\dom}\n_\e^2(\tau)\, dx-\int_{\dom}\n_\e^2(0)\, dx\leq \int_0^\tau\int_{\dom} \n_\e^2\Delta\W_\e+2\n_\e^2\G(p_\e)\, dx dt
\end{equation*}
Passing to the limit $\e\rightarrow 0$ in this inequality, we have
\begin{equation}\label{eq:hej}
\int_{\dom}\overline{\n^2}(\tau)\, dx-\int_{\dom}\n^2_0\, dx\leq \int_0^\tau\int_{\dom} \overline{\n^2\Delta\W}+2\overline{\n^2\G(p)}\, dx dt,
\end{equation}
where $\overline{\n^2}$ denotes the weak limit of $\n_\e^2$ and $\overline{\n^2\Delta \W}$ and $\overline{\n^2 \G(p)}$ are the weak limits of $\n_\e^2\Delta \W_\e$ and $\n_\e^2 \G(p_\e)$ respectively. Letting $\tau\rightarrow 0$ in this inequality, we obtain, thanks to the boundedness of the integrand on the right hand side,
\begin{equation*}
\int_{\dom}\overline{\n^2}(0)\, dx-\int_{\dom}\n^2_0\, dx\leq 0.
\end{equation*}
On the other hand, since $b(\n)=\n^2$ is convex, we have $\overline{\n^2}\geq \n^2$ and hence $\overline{\n^2}(0,x)=\n_0^2(x)$.

We now choose smooth test functions $\varphi_\eps$ approximating $\varphi(t,x)=\mathbf{1}_{[0,\tau]}(t)$, where $\tau\in (0,T]$, in inequality \eqref{eq:faen1} and then pass to the limit in the approximation to obtain the inequality
\begin{equation}\label{eq:faene}
\int_{\dom}\n^2(\tau)\, dx-\int_{\dom}\n^2_0\, dx= \int_{0}^\tau\int_{\dom} (2\n\overline{\n\G(p)}+\n^2\Delta\W)\, dxdt
\end{equation}
Subtracting  \eqref{eq:faene} from \eqref{eq:hej}, we have
\begin{multline}\label{eq:schafseckel2}
\int_{\dom} \left(\overline{\n^2}-\n^2\right)\!(\tau) dx\\
\leq \int_0^\tau\!\!\int_{\dom}\left(2 \overline{\n^2 \G(p)}-2n\overline{\n\G({p})}+\Delta\W \left(\overline{\n^2}-\n^2\right)+\overline{\n^2\Delta \W}-\overline{n^2}\Delta\W\right) dx\, dt.
\end{multline}
%
%
Now using the explicit expression of $\G$, \eqref{G-example}, the first term on the right hand side can be estimated as follows:
\begin{align}\label{eq:estimate1}
\begin{split}
&\int_0^\tau\!\!\int_{\dom}\left(2 \overline{\n^2 \G(p)}-2n\overline{\n\G({p})}\right) dx\, dt\\
&\qquad=2\int_0^\tau\!\!\int_{\dom}\alpha  \left( \overline{\n^2}-\n^2 \right)- \beta \left(\overline{\n^{2+\gamma\theta}} - \n\overline{\n^{1+\gamma\theta}}\right)  dx\, dt\\
&\qquad\leq 2\int_0^\tau\!\!\int_{\dom}\alpha  \left( \overline{\n^2}-\n^2 \right)- \beta \left(\overline{\n^{2+\gamma\theta}} - \overline{\n^{2+\gamma\theta}}\right)  dx\, dt\\
&\qquad\leq 2\alpha \int_0^\tau\!\!\int_{\dom}\left( \overline{\n^2}-\n^2 \right)dx\, dt
\end{split}
\end{align}
where we have used \cite[Lemma 3.35]{Novotny-Stras-2004}, which implies $\n \overline{\n^{1+\gamma\theta}}\leq \overline{\n^{2+\gamma\theta}}$, for the first inequality. 
To estimate the second term on the right hand side, we use that $\Delta \W$ is bounded thanks to Lemma \ref{lem:We} and that $\overline{\n^2}\geq \n^2$ by the convexity of $f(x)=x^2$.
Hence
\begin{equation}\label{eq:helvete}
\int_0^\tau\int_{\dom}\Delta\W\left(\overline{\n^2}-\n^2  \right)dxdt\leq \frac{P_M}{\mu}\int_0^\tau\int_{\dom}\left(\overline{\n^2}-\n^2  \right)dxdt.
\end{equation}

For the last term, we use the following lemma,
\begin{lemma}\label{lem:effve}
The weak limits $(\n,\W,\weakl{p})$ of the sequences $\{(\n_\e,\W_\e,p_\e)\}_{\e>0}$ satisfy for smooth functions $S:\R\rightarrow\R$,
\begin{equation}\label{eq:gugus2}
\int_{\dom}\left(\weakl{S(\n)\Delta\W}- \weakl{S(\n)}\Delta\W\right) dx = \frac{1}{\mu} \int_{\dom}\left(\weakp{p}{S(\n)}- \weakl{pS(\n)}\right)dx
\end{equation}
where $\overline{S(\n)\Delta\W}$, $\overline{S(\n)}$, $\overline{pS(\n)}$ are the weak limits of $S(\n_\e)\Delta\W_\e$, $S(\n_\e)$ and $p_\e S(\n_\e)$ respectively.
\end{lemma}
Applying this lemma to the second term in \eqref{eq:schafseckel2} with $S(\n)=\n^2$, we can estimate it by
\begin{align*}
\int_{0}^\tau\int_{\dom}\left(\weakl{\n^2\Delta\W}- \weakl{\n^2}\Delta\W\right) dx &= \frac{1}{\mu} \int_{\dom}\left(\weakp{p}{\n^2}- \weakl{p \n^2}\right)dxdt\\
&= \frac{1}{\mu} \int_{\dom}\left(\weakp{\n^{\gamma}}{\n^2}- \weakl{ \n^{2+\gamma}}\right)dxdt\\
&\leq 0,
\end{align*}
using that $\weakp{\n^{\gamma}}{\n^2}\leq \weakl{ \n^{2+\gamma}}$ 
(cf.  \cite{Novotny-Stras-2004}). Thus,
\begin{equation*}
\int_{\dom} \left(\weakl{\n^2}-\n^2\right)\!(\tau) dx \leq \left(2\alpha+\frac{P_M}{\mu}\right) \int_0^\tau\!\!\int_{\dom}\left(\weakl{\n^2}-\n^2\right) dx\, dt.
\end{equation*}
Hence Gr\"{o}nwall's inequality implies
\begin{equation*}
\int_{\dom} \left(\weakl{\n^2}-\n^2\right)\!(\tau) dx\leq 0
\end{equation*}
By convexity of the function $f(x)=x^2$ we also have $\n^2\leq \weakl{\n^2}$ almost everywhere and so
\begin{equation*}
\weakl{\n^2}(t,x)=\n^2(t,x)
\end{equation*}
almost everywhere in $\Dom$. Therefore we conclude that the functions $\n_\e$ converge strongly to $\n$ almost everywhere and in particular also $\overline{p}=\n^\gamma$ which means that the limit $(\n,\W,\overline{p})$ is a weak solution of the equations \eqref{HeleShaw}.
\begin{proof}[Proof of Lemma \ref{lem:effve}]
We multiply the equation for $\W_\e$ by $S(\n_\e)$ and integrate over $\dom$,
%
%
\begin{equation*}
  \int_{\dom}\mu\Delta \W_\e\,S(\n_\e)-\W_\e S(\n_\e)\,dx = -\int_{\dom} p_\e S(\n_\e) \,dx. 
\end{equation*}
Passing to the limit $\e\rightarrow 0$, we obtain
\begin{equation}\label{eq:dubel2}
 \int_{\dom}\mu\overline{\Delta\W S(\n)}-\W \overline{S(\n)}\,dx =-\int_{\dom} \overline{ p S(\n)} \,dx.
\end{equation}
On the other hand, using the smooth function $S(\n_\e)$ as a test function in the weak formulation of the limit equation
\begin{equation*}
-\mu\Delta\W+\W= \overline{p},
\end{equation*}
and passing to the limit $\e\rightarrow 0$,
we obtain
\begin{equation*}
\int_{\dom} \mu\Delta\W \overline{S(\n)} -\W\overline{S(\n)}\, dx=-\int_{\dom}  \weakp{p}{S(\n)}\, dx.
\end{equation*}
Combining the last identity with \eqref{eq:dubel2}, we obtain \eqref{eq:gugus2}.
\end{proof}

%
%
%
%
%
%
%
%
%
%
%
%
%
%
%
%
%
%
%
%
%
%
%
%
%
%

\section{Global existence via a numerical approximation}
\label{S4}
We consider the problem in two space dimensions in a rectangular domain, for simplicity we use $\dom=[0,1]^2$, the generalization to other rectangular domains as well as three space dimensions is straightforward but more cumbersome in terms of notation, for this reason we restrict ourself to a square two dimensional domain here. For simplicity, we will also assume $a=1$ in the Brinkman law in \eqref{HeleShaw}. We let $h>0$ the mesh width, and $\Dt$ the time step size. We will determine the necessary ratio between $h$ and $\Dt$ later on. 
For $i,j=1,\dots, N_x$, where $N_x=1/h$, $h$ chosen such that $N_x$ is an integer, we denote grid cells $\cell_{ij}:=((i-1)h,ih]\times ((j-1)h,jh]$ with cell midpoints $x_{i,j}=((i-1/2)h,(j-1/2)h)$. In addition, we denote $t^m=m \Dt$, $m=0,\dots N_T$, where $N_T=T/\Dt$ for some final time $T>0$. The approximation of a function $f$ at grid point $x_{i,j}$ and time $t^m$ will be denoted $f_{i,j}^m$. We also introduce the finite differences,
\begin{align*}
\begin{split}
D^\pm_1 f_{ij}=\pm\frac{f_{i\pm1,j}-f_{i,j}}{h},\quad D^\pm_2 f_{ij}=\pm\frac{f_{i,j\pm1}-f_{i,j}}{h},\quad D^\pm_t f^m=\pm\frac{f^{m\pm 1}-f^m}{\Dt}.
\end{split}
\end{align*}
and define the discrete Laplacian, divergence and gradient operators based on these,
\begin{equation*}
\Grad_h^{\pm}:=(D^\pm_1,D^\pm_2)^t,\quad \Div_h^{\pm} f_{i,j}=D^{\pm}_1 f^{(1)}_{i,j}+ D^{\pm}_2 f^{(2)}_{i,j},\quad  \Delta_h:= \Div_h^{\pm}\Grad_h^{\mp}.
\end{equation*}
For ease of notation, we also let $u_{i+1/2,j}$ and $v_{i,j+1/2}$ denote the discrete velocities in the transport equation, specifically, given $\W_{i,j}$, we let
\begin{equation}\label{eq:discv}
u_{i+1/2,j}:=D^+_1 \W_{i,j},\quad v_{i,j+1/2}:= D^+_2 \W_{i,j}.
\end{equation}
\subsection{An explicit finite difference scheme} Given $(\n_{i,j}^m,\W_{i,j}^m)$ at time step $m$, we define the quantities $(\n_{i,j}^{m+1},\W_{i,j}^{m+1})$ at the next time step by
\begin{subequations}\label{eq:expscheme1}
\begin{align}
\label{seq2:Wscheme}
-\mu \Delta_h \W^{m}_{i,j}+ \W_{i,j}^{m}&= p_{i,j}^{m},\\
p_{i,j}^{m}&:= |\n_{i,j}^{m}|^\gamma,\\
\label{seq1:nscheme}
D_t^+ \n_{i,j}^m+ D^-_1 F^{(1)}_{i+1/2,j}(u^m,n^m)+ D^-_2 F^{(2)}_{i,j+1/2}(v^m,n^m) &= \n_{i,j}^m \G (p_{i,j}^m),
\end{align}
\end{subequations}

where $p_{i,j}=(\n_{i,j})^\gamma$ and the fluxes $F^{(j)}$, $j=1,2$ are defined by
\begin{align}\label{eq:fluxes}
\begin{split}
F^{(1)}_{i+1/2,j}(u^m,n^m)&= - u_{i+1/2,j}^m \frac{\n_{i,j}^m+\n_{i+1,j}^m}{2}-\frac{h}{2} |u_{i+1/2,j}| D^+_1 \n_{i,j}^m\\
F^{(2)}_ {i,j+1/2}(v^m,n^m)& =-  v_{i,j+1/2}^m \frac{\n_{i,j}^m+\n_{i,j+1}^m}{2}-\frac{h}{2} |v_{i,j+1/2}| D^+_2 \n_{i,j}^m.
\end{split}
\end{align}
We use homogeneous Neumann or periodic boundary conditions for both variables:
\begin{align*}
\n_{0,j}^m&=\n_{1,j}^m,\qquad \n_{N_x+1,j}^m=\n_{N_x,j}^m, &j=1,\dots,N_x,\\
\n_{i,0}^m&=\n_{i,1}^m,\qquad \n_{i,N_x+1}^m=\n_{i,N_x}^m, &i=1,\dots,N_x,\\
\W_{0,j}^m&=\W_{1,j}^m,\qquad \W_{N_x+1,j}^m=\W^m_{N_x,j}, &j=1,\dots,N_x,\\
\W_{i,0}^m&=\W_{i,1}^m,\qquad \W_{i,N_x+1}^m=\W_{i,N_x}^m, &i=1,\dots,N_x.
\end{align*}
The initial condition we approximate taking averages over the cells,
\begin{equation*}
\n_{i,j}^0=\frac{1}{|\cell_{ij}|}\int_{\cell_{ij}}\!\! \n_0(x)\, dx,\quad p^0_{i,j}=|\n_{i,j}^0|^\gamma,\quad i,j=1,\dots, N_x.
\end{equation*}
\subsection{Estimates on approximations}
In the following, we will prove estimates on the discrete quantities $(\n_{i,j}^{m},\W_{i,j}^{m})$ obtained using the scheme \eqref{eq:discv}--\eqref{eq:fluxes}.
We therefore define the piecewise constant functions
\begin{equation}\label{eq:interpolation}
f_h(t,x)=\sum_{m=0}^{N_T}\sum_{i,j=1}^{N_x} f_{i,j}^m\, \mathbf{1}_{\cell_{ij}}(x)\mathbf{1}_{[t^m,t^{m+1})}(t),\quad (t,x)\in [0,T]\times\dom,\\
\end{equation}
where $f\in \{\n,\W,p\}$. 
We first prove that $\n_h$ stays nonnegative and uniformly bounded from above.
\begin{lemma}\label{lem:linfn_h}
If $0\leq \n_{i,j}^0\leq \n_{\infty}:=P_M^{1/\gamma}<\infty$ uniformly in $h>0$ and the timestep $\Dt$ satisfies the CFL condition
\begin{equation}\label{eq:CFL1}
\Dt\leq \min\left\{\frac{h}{8\max_{ij} |\Grad_h \W_{i,j}^m|+h\G^\infty},\frac{\mu}{4\gamma \nmax^\gamma}\right\}
\end{equation}
{\normalfont(}where $\G^\infty:=\max_{s\in\R^+}{\G(s)}${\normalfont)},
then for any $t>0$, the functions $\n_h(t,\cdot)$ are uniformly {\normalfont(}in $h>0${\normalfont)} bounded and nonnegative,  specifically, defining $\nmax=\n_\infty+4\Dt \sup_{s\geq 0} \left( s^{1/\gamma}\G(s)\right)$, we have for all $m\geq 0$,
\begin{equation*}
0\leq \min_{i,j}\n_{i,j}^m\leq \max_{i,j} \n_{i,j}^{m}\leq \nmax.
\end{equation*}
\end{lemma}
\begin{proof}
The proof goes by induction on the timestep $m$. Clearly, by the assumptions, we have $0\leq \n_{i,j}^0\leq \nmax$. For the induction step we therefore assume that this holds for timestep $m>0$ and show that it implies the nonnegativity and boundedness at timestep $m+1$.

We first show that the $\W_{i,j}^m$ are bounded in terms of the $p_{i,j}^m$. To do so, let us assume it has a local maximum $\W_{\ihat,\jhat}^m$ in a cell $\cell_{\ihat \jhat}$, for some $\ihat,\jhat\in \{1,\dots,N_x\}$. Then 
$$
D^+_k \W_{\ihat,\jhat}^m\leq 0,\quad -D^-_k \W_{\ihat,\jhat}^m\leq 0, \quad k=1,2,
$$
(if $\ihat$ or $\jhat\in\{1,N_x\}$, then because of the Neumann boundary conditions, the forward/backward difference in direction of the boundary is zero and thus the previous inequality is true as well). Hence
$$
\Delta_h \W_{\ihat,\jhat}^m=\frac{1}{h}\sum_{k=1}^2 \left(D^+_k \W_{\ihat,\jhat}^m- D^-_k \W_{\ihat,\jhat}^m\right)\leq 0.
$$

Therefore,
\begin{equation*}
\W_{\ihat,\jhat}^m=p_{\ihat,\jhat}^m+\frac{1}{\mu}\Delta_h\W_{\ihat,\jhat}^m\leq p_{\ihat,\jhat}^m\leq \max_{i,j}|\n_{i,j}^m|^\gamma.
\end{equation*}
Similarly, at a local minimum $\W_{\ihat,\jhat}^m$ of $\W_h$, we have
$$
D^+_k \W_{\ihat,\jhat}^m\geq 0,\quad -D^-_k \W_{\ihat,\jhat}^m\geq 0, \quad k=1,2,
$$
and hence
$$
\Delta_h \W_{\ihat,\jhat}^m=\frac{1}{h}\sum_{k=1}^2 \left(D^+_k \W_{\ihat,\jhat}^m- D^-_k \W_{\ihat,\jhat}^m\right)\geq  0,
$$
which implies
\begin{equation*}
\W_{\ihat,\jhat}^m=p_{\ihat,\jhat}^m+\frac{1}{\mu}\Delta_h\W_{\ihat,\jhat}^m\geq p_{\ihat,\jhat}^m\geq \min_{i,j} |\n_{i,j}^m|^\gamma\geq 0.
\end{equation*}
Thus,
\begin{equation}\label{eq:Wbound}
0\leq \W_h\leq \max_{i,j}|\n_{i,j}^m|^\gamma. 
\end{equation}
Now we rewrite the scheme \eqref{seq1:nscheme} as
\begin{equation}\label{eq:n_convex}
\n_{i,j}^{m+1}=\left(\alpha_{i,j}^{(1),m}+\alpha_{i,j}^{(2),m}\right) \n_{i,j}^m+ \beta_{i,j}^m \n_{i+1,j}^m+ \zeta_{i,j}^m \n_{i-1,j}^m+ \eta_{i,j}^m \n_{i,j+1}^m+ \theta_{i,j}^m \n_{i,j-1}^m
\end{equation}
where
\begin{align*}
\begin{split}
\alpha_{i,j}^{(1),m}&= 1-\frac{\Dt}{2h}\bigl[ (|u_{i+1/2,j}^m|+u_{i+1/2,j}^m)+(|u_{i-1/2,j}^m|-u_{i-1/2,j}^m) \\
&\hphantom{= 1-\frac{\Dt}{h}\bigl[ (} +(|v_{i,j+1/2}^m|+v_{i,j+1/2}^m) +(|v_{i,j-1/2}^m|-v_{i,j-1/2}^m)\bigr]\\
\alpha_{i,j}^{(2),m}&= \Dt \,\G(p_{i,j}^m)+\frac{\Dt}{h}\left[u_{i+1/2,j}^m-u_{i-1/2,j}^m+v_{i,j+1/2}^m-v_{i,j-1/2}^m\right]\\
\beta_{i,j}^m&=\frac{\Dt}{2h} \left( u_{i+1/2,j}^m+| u_{i+1/2,j}^m|\right)\\
\zeta_{i,j}^m&=\frac{\Dt}{2h}\left(|u_{i-1/2,j}^m|-u_{i-1/2,j}^m\right)\\
\eta_{i,j}^m&=\frac{\Dt}{2h} \left( v_{i,j+1/2}^m+| v_{i,j+1/2}^m|\right)\\
\theta_{i,j}^m&=\frac{\Dt}{2h}\left(|v_{i,j-1/2}^m|-v_{i,j-1/2}^m\right)\\
\end{split}
\end{align*}
We note that $\beta_{i,j}^m, \zeta_{i,j}^m, \eta_{i,j}^m, \theta_{i,j}^m\geq 0$, and that under the CFL-condition \eqref{eq:CFL1}, also $\alpha_{i,j}^{(1),m}+\alpha_{i,j}^{(2),m}\geq 0$. Hence, assuming that $\n_{i,j}^m\geq 0$ for all $i,j$, we have
\begin{equation*}
\begin{split}
\n_{i,j}^{m+1}&\geq \left(\beta_{i,j}^m+\zeta_{i,j}^m+\eta_{i,j}^m+\theta_{i,j}^m\right) \min\{\n_{i+1,j}^m,\n_{i-1,j}^m,\n^m_{i,j+1},\n^m_{i,j-1} \}\\
&\quad +\left(\alpha_{i,j}^{(1),m}+\alpha_{i,j}^{(2),m}\right)\n_{i,j}^m\\
&\geq 0.
\end{split}
\end{equation*}
We proceed to showing the boundedness of $\n_h$. Thanks to the CFL-condition \eqref{eq:CFL1}, we have
\begin{equation*}
\alpha_{i,j}^{(1),m}\geq \frac{1}{2},\quad \beta_{i,j}^m, \zeta_{i,j}^m, \eta_{i,j}^m, \theta_{i,j}^m\leq \frac{1}{8}.
\end{equation*}
Moreover, $\alpha_{i,j}^{(1),m}+\beta_{i,j}^m+ \zeta_{i,j}^m+ \eta_{i,j}^m+ \theta_{i,j}^m=1$. Using the induction hypothesis that $\n_{i,j}^m\leq \nmax$ for all $i,j$ and the nonnegativity of $\n_h$ which we have just proved, we can estimate $\n_{i,j}^{m+1}$:
\begin{align}\label{eq:waeck}
\begin{split}
\n_{i,j}^{m+1}&\leq \left(\alpha_{i,j}^{(1),m}+\alpha_{i,j}^{(2),m}\right) \n_{i,j}^m+\left(\beta_{i,j}^m+ \zeta_{i,j}^m+ \eta_{i,j}^m+ \theta_{i,j}^m\right)\nmax\\
&\leq \left(\frac{1}{2}+\alpha_{i,j}^{(2),m}\right)\n_{i,j}^m+\frac{1}{2}\nmax\\
&=\nmax-\frac{1}{2}\left(\nmax-\n_{i,j}^m\right)+\alpha_{i,j}^{(2),m}\n_{i,j}^m
\end{split}
\end{align}
We can rewrite and bound $\alpha_{i,j}^{(2),m}$ using the equation for $\W_{i,j}^m$, \eqref{seq2:Wscheme},
\begin{align*}
\alpha_{i,j}^{(2),m}&=\Dt\left(\G(p_{i,j}^m)+\Delta_h \W_{i,j}^m\right)\\
&=\Dt\left(\G(p_{i,j}^m)+\frac{1}{\mu}\left(\W_{i,j}^m-p_{i,j}^m\right)\right)\\
&\leq \Dt\left(\G(p_{i,j}^m)+\frac{1}{\mu}\left(\nmax^\gamma-|\n_{i,j}^m|^\gamma\right)\right)\\
&\leq \Dt\left(\G(p_{i,j}^m)+\frac{\gamma\,\nmax^{\gamma-1}}{\mu}\left(\nmax-\n_{i,j}^m\right)\right)\\
&\leq \Dt \G(p_{i,j}^m)+\frac{1}{4\nmax} (\nmax-\n_{i,j}^m) ,
\end{align*}
where we have used \eqref{eq:Wbound} for the first inequality, that $f(a)-f(b)=f'(\widetilde{a})(a-b)$ for some intermediate value $\widetilde{a}\in [b,a]$, with $f(a)=a^\gamma$, for the second inequality and the CFL-condition for the last inequality. Now going back to \eqref{eq:waeck} and inserting this there, we obtain,
\begin{align}\label{eq:pokker}
\n_{i,j}^{m+1}&\leq \nmax-\frac{1}{2}\left(\nmax-\n_{i,j}^m\right)+\left(\Dt \G(p_{i,j}^m)+\frac{1}{4\nmax} (\nmax-\n_{i,j}^m)\right)\n_{i,j}^m\notag\\
&\leq \frac{3}{4}\nmax+\frac{1}{4}\n_{i,j}^m+\Dt \n_{i,j}^m\G(p_{i,j}^m)
\end{align}
If $\n_{i,j}^m\geq \n_{\infty}$ then $\G(p_{i,j}^m)\leq 0$ and hence the expression in \eqref{eq:pokker} is bounded by $\nmax$. On the other hand, if $\n_{i,j}^m\leq \n_{\infty}$, we can bound it by
\begin{align*}
\n_{i,j}^{m+1}&\leq \frac{3}{4}\nmax+\frac{1}{4}\n_{i,j}^m+\Dt \n_{i,j}^m\G(p_{i,j}^m)\\
&\leq \frac{3}{4}\nmax+\frac{1}{4}\left(\n_\infty+4\Dt \sup_{s\geq 0}\left( s^{1/\gamma} \G(s)\right)\right)\\
&=\nmax
\end{align*}
where we used the definition of $\nmax$ for the last equality. This proves that $\n_{i,j}^{m+1}\leq \nmax$ for all $i,j$ if the same holds already for the $\n_{i,j}^m$.
\end{proof}
\begin{remark}
The estimates in the proof of the previous lemma are very coarse and therefore one can use a much larger CFL-condition than \eqref{eq:CFL1} in practice. Also note that $\nmax\rightarrow \n_\infty$ when $\Dt\rightarrow 0$. 
\end{remark}

\subsubsection{Estimates on the discrete potential $\W_h$}
\begin{lemma}\label{lem:Wh}
We have that 
$$
\W_h,\Grad_h\W_h,\Grad^2_h \W_h\subset L^\infty([0,T];L^2(\dom)),
$$ uniformly in $h>0$, where $\Grad_h:=\Grad_h^\pm$ and $\Grad_h^2:=\Grad_h^\mp\Grad_h^\pm$ and 
$$
\W_h, \Delta_h \W_h\subset L^\infty(\Dom)),
$$
uniformly in $h>0$ as well.
\end{lemma}
\begin{proof}
To obtain the $L^2$-estimates, we square the equation for the potential $\W_h$, \eqref{seq2:Wscheme} and sum over all $i,j$,
\begin{equation*}
\mu^2\sum_{i,j=1}^{N_x} |\Delta_h \W^{m}_{i,j}|^2-2\mu\sum_{i,j=1}^{N_x}\W_{i,j}^{m}\Delta_{h} \W_{i,j}^{m}+\sum_{i,j=1}^{N_x} |\W_{i,j}^{m}|^2= \sum_{i,j=1}^{N_x} |n_{i,j}^{m}|^{2\gamma}. 
\end{equation*}
Using summation by parts and that $\W$ satisfies either periodic or homogeneous Neumann boundary conditions, we obtain
\begin{equation*}
\mu^2\sum_{i,j=1}^{N_x} |\Grad_h^2 \W^{m}_{i,j}|^2+2\mu\sum_{i,j=1}^{N_x}|\Grad_h \W_{i,j}^{m}|^2+\sum_{i,j=1}^{N_x} |\W_{i,j}^{m}|^2= \sum_{i,j=1}^{N_x} |n_{i,j}^{m}|^{2\gamma}.
\end{equation*}
From the previous estimates, we know that $\n_h\in L^{\infty}([0,T]\times\dom)$ uniformly in $h>0$ and therefore also uniformly bounded in any other $L^p$-space, which implies together with the above identity, that $\W_h,\Grad_h \W_h, \Grad_h^2\W_h\in L^2([0,T]\times \dom)$.
That $\W_h$ is uniformly bounded follows from \eqref{eq:Wbound} and the uniform bound on $\n_h$ which was proved in the previous Lemma \ref{lem:linfn_h}.

Using this and the uniform boundedness of the pressure, we conclude by \eqref{seq2:Wscheme} that also $\Delta_h \W_h$ is uniformly bounded.
\end{proof}
\begin{remark}
Using the discrete Gagliardo-Nirenberg-Sobolev inequality, \cite[Thm.~3.4]{Bessemoulin2014}, we obtain that $\Grad_h\W_h\in L^\infty([0,T];L^q(\dom))$ for $1\leq q< q^*=2d/(d-2)$.
\end{remark}
\subsection{Discrete entropy inequalities for $\n_h$}
To prove strong convergence of the approximating sequence $\{(\n_h,\W_h)\}_{h>0}$, it will be useful to derive entropy inequalities for $\n_h$. To this end, the following lemma will be useful:
\begin{lemma}\label{lem:fn}
Let $f:\R\rightarrow \R$ be a smooth convex function and assume that $\Dt$ satisfies the CFL-condition
\begin{equation}\label{eq:CFL2}
\Dt\leq \min\left\{\frac{h}{16\max_{i,j}|\Grad_h \W_{i,j}^m|},\frac{h}{8\max_{ij} |\Grad_h \W_{i,j}^m|+h\, \G^\infty },
\frac{\mu}{4\gamma \nmax^\gamma}\right\}
\end{equation}
Denote $f_{i,j}^m:=f(\n_{i,j}^m)$ and $f_h$ a piecewise constant interpolation of it as in \eqref{eq:interpolation}. Then $f_{i,j}^m$ satisfies the following identity
\begin{align}
\label{eq:t1}
D_t f_{i,j}^m= & \frac{1}{2} D^-_1\left( u_{i+1/2,j}^m \left(f_{i,j}^m+f_{i+1,j}^m\right)\right) + \frac{1}{2} D_2^-\left( v_{i,j+1/2}^m \left(f_{i,j}^m +f_{i,j+1}^m\right)\right)\\
\label{eq:t2}
&\hphantom{=}+\frac{h}{4} D^-_1 \left[f'(\n_{i,j}^m)|u_{i+1/2,j}^m| D^+_1 \n_{i,j}^m\right] +\frac{h}{4} D^-_2 \left[f'(\n_{i,j}^m)|v_{i,j+1/2}^m| D^+_2 \n_{i,j}^m\right]\\
\label{eq:t3}
&\hphantom{=}+\frac{h}{4} D^+_1 \left[f'(\n_{i,j}^m)|u_{i-1/2,j}^m| D^-_1 \n_{i,j}^m\right] +\frac{h}{4} D^+_2 \left[f'(\n_{i,j}^m)|v_{i,j-1/2}^m| D^-_2 \n_{i,j}^m\right]\\
\label{eq:t4}
&\hphantom{=}-\frac{h^2}{4}D^-_1\left[ f''(\widetilde{n}_{i+1/2,j}^m) u_{i+1/2,j}^m |D^+_1 \n_{i,j}^m|^2\right]\\
\label{eq:t5}
&\hphantom{=}-\frac{h^2}{4}D^-_2\left[ f''(\widetilde{n}_{i,j+1/2}^m) v_{i,j+1/2}^m |D^+_2 \n_{i,j}^m|^2\right]\\
\label{eq:t6}
&\hphantom{=} -\frac{h}{4}f''(\widehat{\n}_{i-1/2,j}^m) |u_{i-1/2,j}^m||D^-_1 \n_{i,j}^m|^2-\frac{h}{4}f''(\widehat{\n}_{i,j-1/2}^m) |v_{i,j-1/2}^m||D^-_2 \n_{i,j}^m|^2\\
\label{eq:t7}
&\hphantom{=} -\frac{h}{4}f''(\widehat{\n}_{i+1/2,j}^m) |u_{i+1/2,j}^m||D^+_1 \n_{i,j}^m|^2-\frac{h}{4}f''(\widehat{\n}_{i,j+1/2}^m) |v_{i,j+1/2}^m||D^+_2 \n_{i,j}^m|^2\\
\label{eq:t8}
&\hphantom{=}+(f'(\n_{i,j}^m)\n_{i,j}^m-f_{i,j}^m) \Delta_h \W_{i,j}^m+f'(\n_{i,j}^m) \n_{i,j}^m \G(p_{i,j}^m)\\
\label{eq:t9}
&\hphantom{=}+\frac{\Dt}{2} f''(\widetilde{\n}^{m+1/2}_{i,j}) |D_t^+ \n_{i,j}^m|^2,
\end{align}
where $\widetilde{\n}_{i\pm1/2,j}^m,\widehat{\n}_{i\pm1/2,j}^m\in [\min\{\n_{i,j}^m,\n_{i\pm1,j}^m\},\max\{\n_{i,j}^m,\n_{i\pm1,j}^m\}]$, $\widetilde{\n}_{i,j\pm1/2}^m,\widehat{\n}_{i,j\pm1/2}^m\in [\min\{\n_{i,j}^m,\n_{i,j\pm1}^m\},\max\{\n_{i,j}^m,\n_{i,j\pm1}^m\}]$ and $\widetilde{\n}^{m+1/2}_{i,j}\in  [\min\{\n_{i,j}^m,\n_{i,j}^{m+1}\},\max\{\n_{i,j}^m,\n_{i,j}^{m+1}\}]$ and where the term \eqref{eq:t8} is uniformly bounded and the terms \eqref{eq:t6} -- \eqref{eq:t7} and \eqref{eq:t9} satisfy
\begin{equation}\label{eq:dritt}
\begin{split}
&\frac{h^{d+1}\Dt}{2}\sum_{m=0}^{N_T}\sum_{i,j}f''(\widehat{\n}_{i+1/2,j}^m) |u_{i+1/2,j}^m||D^+_1 \n_{i,j}^m|^2\leq C,\\
&\frac{h^{d+1}\Dt}{2}\sum_{m=0}^{N_T}\sum_{i,j}  f''(\widehat{\n}_{i,j+1/2}^m) |v_{i,j+1/2}^m||D^+_2 \n_{i,j}^m|^2\leq C,\\
&\frac{h^{d}\Dt^2}{2}\sum_{m=0}^{N_T}\sum_{i,j} f''(\widetilde{\n}^{m+1/2}_{i,j}) |D_t^+ \n_{i,j}^m|^2\leq C,
\end{split}
\end{equation}
In particular, this implies that the piecewise constant interpolation $D_t^+f_h$ is of the form $D_t^+f_h=g_h+k_h$ where $g_h\in L^1([0,T]\times \dom)$ and $k_h\in L^\infty([0,T]; W^{-1,q}(\dom))$ for any $1\leq q<\infty$ if $d=2$ and for $1\leq q \leq q^*=2d/(d-2)$ if $d>2$, uniformly in $h>0$. 
\end{lemma}
\begin{proof}
We first rewrite the scheme for $\n_{i,j}^m$ as
\begin{align}\label{eq:sternefoifi}
\begin{split}
D_t^+ \n_{i,j}^m&= \frac{1}{2} u_{i+1/2,j}^m D^+_1 \n_{i,j}^m+\frac{1}{2} u_{i-1/2,j}^m D^-_1 \n_{i,j}^m\\
&\hphantom{=} +\frac{1}{2} v_{i,j+1/2}^m D^+_2 \n_{i,j}^m+\frac{1}{2} v_{i,j-1/2}^m D^-_2 \n_{i,j}^m\\
&\hphantom{=}+\frac{h}{2} D^-_1 \left[|u_{i+1/2,j}^m| D^+_1 \n_{i,j}^m\right] +\frac{h}{2} D^-_2 \left[|v_{i,j+1/2}^m| D^+_2 \n_{i,j}^m\right]\\
&\hphantom{=}+\n_{i,j}^m \Delta_h \W_{i,j}^m+\n_{i,j}^m \G(p_{i,j}^m).
\end{split}
\end{align}
%
%
%
%
%
Then, using the Taylor expansion,
\begin{equation*}
f(b)-f(a)=f'(a)(b-a)+f''(\widetilde{a})\frac{(a-b)^2}{2},
\end{equation*}
where $\widetilde{a}\in [\min\{a,b\},\max\{a,b\}]$, we can write
\begin{equation*}
\begin{split}
D_t^+ f_{i,j}^m&=f'(\n_{i,j}^m) D_t^+ \n_{i,j}^m+\frac{\Dt}{2} f''(\widetilde{\n}^{m+1/2}_{i,j}) |D_t^+ \n_{i,j}^m|^2\\
D_1^\pm f_{i,j}^m&=f'(\n_{i,j}^m) D_1^\pm \n_{i,j}^m\pm\frac{h}{2} f''(\widetilde{\n}_{i\pm 1/2,j}^m) |D_1^\pm \n_{i,j}^m|^2\\
D_2^\pm f_{i,j}^m&=f'(\n_{i,j}^m) D_2^\pm \n_{i,j}^m\pm\frac{h}{2} f''(\widetilde{\n}_{i,j\pm 1/2}^m) |D_2^\pm \n_{i,j}^m|^2\\
D_1^\pm f'(\n_{i,j}^m)&=f''(\widehat{\n}_{i\pm 1/2,j}^m) D_1^\pm \n_{i,j}^m\\
D_2^\pm f'(\n_{i,j}^m)&=f''(\widehat{\n}_{i,j\pm 1/2}^m) D_2^\pm \n_{i,j}^m,
\end{split}
\end{equation*}
where $\widetilde{\n}^{m+1/2}_{i,j}, \widetilde{\n}_{i\pm 1/2,j}^m, \widetilde{\n}_{i,j\pm 1/2}^m, \widehat{\n}_{i\pm 1/2,j}^m$ and $\widehat{\n}_{i,j\pm 1/2}^m$ are intermediate values. Hence, multiplying equation \eqref{eq:sternefoifi} by $f'(\n_{i,j}^m)$, it becomes

\begin{align*}
D_t^+ f_{i,j}^m &= \frac{\Dt}{2} f''(\widetilde{\n}^{m+1/2}_{i,j}) |D_t^+ \n_{i,j}^m|^2\\
&\hphantom{=}+ \frac{1}{2} u_{i+1/2,j}^m D^+_1 f_{i,j}^m -\frac{h}{4} f''(\widetilde{n}_{i+1/2,j}^m) u_{i+1/2,j}^m |D^+_1 \n_{i,j}^m|^2\\
&\hphantom{=}+ \frac{1}{2} u_{i-1/2,j}^m D^-_1 f_{i,j}^m +\frac{h}{4} f''(\widetilde{n}_{i-1/2,j}^m) u_{i-1/2,j}^m |D^-_1 \n_{i,j}^m|^2\\
&\hphantom{=}+ \frac{1}{2} v_{i,j+1/2}^m D^+_2 f_{i,j}^m -\frac{h}{4} f''(\widetilde{n}_{i,j+1/2}^m) v_{i,j+1/2}^m |D^+_2 \n_{i,j}^m|^2\\
&\hphantom{=}+ \frac{1}{2} v_{i,j-1/2}^m D^-_2 f_{i,j}^m +\frac{h}{4} f''(\widetilde{n}_{i,j-1/2}^m) v_{i,j-1/2}^m |D^-_2 \n_{i,j}^m|^2\\
&\hphantom{=}+\frac{h}{4} D^-_1 \left[f'(\n_{i,j}^m)|u_{i+1/2,j}^m| D^+_1 \n_{i,j}^m\right] -\frac{h}{4}f''(\widehat{\n}_{i-1/2,j}^m) |u_{i-1/2,j}^m||D^-_1 \n_{i,j}^m|^2\\
&\hphantom{=}+\frac{h}{4} D^-_2 \left[f'(\n_{i,j}^m)|v_{i,j+1/2}^m| D^+_2 \n_{i,j}^m\right] -\frac{h}{4}f''(\widehat{\n}_{i,j-1/2}^m) |v_{i,j-1/2}^m||D^-_2 \n_{i,j}^m|^2\\
&\hphantom{=}+\frac{h}{4} D^+_1 \left[f'(\n_{i,j}^m)|u_{i-1/2,j}^m| D^-_1 \n_{i,j}^m\right] -\frac{h}{4}f''(\widehat{\n}_{i+1/2,j}^m) |u_{i+1/2,j}^m||D^+_1 \n_{i,j}^m|^2\\
&\hphantom{=}+\frac{h}{4} D^+_2 \left[f'(\n_{i,j}^m)|v_{i,j-1/2}^m| D^-_2 \n_{i,j}^m\right] -\frac{h}{4}f''(\widehat{\n}_{i,j+1/2}^m) |v_{i,j+1/2}^m||D^+_2 \n_{i,j}^m|^2\\
&\hphantom{=}+f'(\n_{i,j}^m)\n_{i,j}^m \Delta_h \W_{i,j}^m+f'(\n_{i,j}^m) \n_{i,j}^m \G(p_{i,j}^m)\\
&= \frac{\Dt}{2} f''(\widetilde{\n}^{m+1/2}_{i,j}) |D_t^+ \n_{i,j}^m|^2\\
&\hphantom{=}+ \frac{1}{2} D^-_1\left( u_{i+1/2,j}^m \left(f_{i,j}^m+f_{i+1,j}^m\right)\right) + \frac{1}{2} D_2^-\left( v_{i,j+1/2}^m \left(f_{i,j}^m +f_{i,j+1}^m\right)\right)\\
&\hphantom{=}+\frac{h}{4} D^-_1 \left[f'(\n_{i,j}^m)|u_{i+1/2,j}^m| D^+_1 \n_{i,j}^m\right] +\frac{h}{4} D^-_2 \left[f'(\n_{i,j}^m)|v_{i,j+1/2}^m| D^+_2 \n_{i,j}^m\right]\\
&\hphantom{=}+\frac{h}{4} D^+_1 \left[f'(\n_{i,j}^m)|u_{i-1/2,j}^m| D^-_1 \n_{i,j}^m\right] +\frac{h}{4} D^+_2 \left[f'(\n_{i,j}^m)|v_{i,j-1/2}^m| D^-_2 \n_{i,j}^m\right]\\
&\hphantom{=}-\frac{h^2}{4}D^-_1\left[ f''(\widetilde{n}_{i+1/2,j}^m) u_{i+1/2,j}^m |D^+_1 \n_{i,j}^m|^2\right]\\
&\hphantom{=}-\frac{h^2}{4}D^-_2\left[ f''(\widetilde{n}_{i,j+1/2}^m) v_{i,j+1/2}^m |D^+_2 \n_{i,j}^m|^2\right]\\
&\hphantom{=} -\frac{h}{4}f''(\widehat{\n}_{i-1/2,j}^m) |u_{i-1/2,j}^m||D^-_1 \n_{i,j}^m|^2-\frac{h}{4}f''(\widehat{\n}_{i,j-1/2}^m) |v_{i,j-1/2}^m||D^-_2 \n_{i,j}^m|^2\\
&\hphantom{=} -\frac{h}{4}f''(\widehat{\n}_{i+1/2,j}^m) |u_{i+1/2,j}^m||D^+_1 \n_{i,j}^m|^2-\frac{h}{4}f''(\widehat{\n}_{i,j+1/2}^m) |v_{i,j+1/2}^m||D^+_2 \n_{i,j}^m|^2\\
&\hphantom{=}+(f'(\n_{i,j}^m)\n_{i,j}^m-f_{i,j}^m) \Delta_h \W_{i,j}^m+f'(\n_{i,j}^m) \n_{i,j}^m \G(p_{i,j}^m).
\end{align*}
which implies \eqref{eq:t1}--\eqref{eq:t9}. In particular, for $f(x)=x^2$, this becomes
\begin{align}\label{eq:L2ent}
\begin{split}
D_t^+ f_{i,j}^m &= \Dt |D_t^+ \n_{i,j}^m|^2\\
&\hphantom{=}+ \frac{1}{2} D^+_1\left( u_{i-1/2,j}^m  \left(f_{i,j}^m+f_{i-1,j}\right)\right)+ \frac{1}{2}D^+_2\left( v_{i,j-1/2}^m\left( f_{i,j}^m+f_{i,j-1}\right)\right)\\
&\hphantom{=} -\frac{h^2}{2} D^-_1\left[u_{i+1/2,j}^m |D^+_1 \n_{i,j}^m|^2\right]-\frac{h^2}{2}D^-_2\left[ v_{i,j+1/2}^m |D^+_2 \n_{i,j}^m|^2\right]\\
&\hphantom{=}+\frac{h}{2} D^-_1 \left[\n_{i,j}^m|u_{i+1/2,j}^m| D^+_1 \n_{i,j}^m\right]+\frac{h}{2} D^-_2 \left[\n_{i,j}^m|v_{i,j+1/2}^m| D^+_2 \n_{i,j}^m\right]  \\
&\hphantom{=}+\frac{h}{2} D^+_1 \left[\n_{i,j}^m|u_{i-1/2,j}^m| D^-_1 \n_{i,j}^m\right] +\frac{h}{2} D^+_2 \left[\n_{i,j}^m|v_{i,j-1/2}^m| D^-_2 \n_{i,j}^m\right]\\
&\hphantom{=}-\frac{h}{2} |u_{i-1/2,j}^m||D^-_1 \n_{i,j}^m|^2-\frac{h}{2}|v_{i,j-1/2}^m||D^-_2 \n_{i,j}^m|^2\\
&\hphantom{=} -\frac{h}{2} |u_{i+1/2,j}^m||D^+_1 \n_{i,j}^m|^2-\frac{h}{2}|v_{i,j+1/2}^m||D^+_2 \n_{i,j}^m|^2\\
&\hphantom{=}+ f_{i,j}^m \Delta_h \W_{i,j}^m+2 f_{i,j}^m \G(p_{i,j}^m),
\end{split}
\end{align}
We estimate the first term on the right hand side of the inequality inserting \eqref{eq:sternefoifi},
\begin{align*}
|D_t^+ \n_{i,j}^m|^2&\leq 2 \biggl| \frac{1}{2} u_{i+1/2,j}^m D^+_1 \n_{i,j}^m+\frac{1}{2} u_{i-1/2,j}^m D^-_1 \n_{i,j}^m +\frac{1}{2} v_{i,j+1/2}^m D^+_2 \n_{i,j}^m\\
&\hphantom{\leq 2 \biggl| }+\frac{1}{2} v_{i,j-1/2}^m D^-_2 \n_{i,j}^m+\frac{h}{2} D^-_1 \left[|u_{i+1/2,j}^m| D^+_1 \n_{i,j}^m\right] +\frac{h}{2} D^-_2 \left[|v_{i,j+1/2}^m| D^+_2 \n_{i,j}^m\right]\biggr|^2 \\
&\hphantom{\leq }+2\bigl|\n_{i,j}^m \Delta_h \W_{i,j}^m+\n_{i,j}^m \G(p_{i,j}^m)\bigr|^2\\
&\leq 4\biggl|\frac{1}{2} u_{i+1/2,j}^m D^+_1 \n_{i,j}^m+\frac{1}{2} u_{i-1/2,j}^m D^-_1 \n_{i,j}^m+\frac{h}{2} D^-_1 \left[|u_{i+1/2,j}^m| D^+_1 \n_{i,j}^m\right]\biggr|^2\\
&\hphantom{\leq }+4\biggl|\frac{1}{2} v_{i,j+1/2}^m D^+_2 \n_{i,j}^m+\frac{1}{2} v_{i,j-1/2}^m D^-_2 \n_{i,j}^m+\frac{h}{2} D^-_2 \left[|v_{i,j+1/2}^m| D^+_2 \n_{i,j}^m\right]\biggr|^2\\
&\hphantom{\leq } +2\bigl|\n_{i,j}^m \Delta_h \W_{i,j}^m+\n_{i,j}^m \G(p_{i,j}^m)\bigr|^2\\
&\leq  8 \big| u_{i+1/2,j}^m D^+_1 \n_{i,j}^m\big|^2+8\big| u_{i-1/2,j}^m D^-_1 \n_{i,j}^m\big|^2+ 8 \big| v_{i,j+1/2}^m D^+_2 \n_{i,j}^m\big|^2\\
&\hphantom{\leq }+8\big| u_{i,j-1/2}^m D^-_2 \n_{i,j}^m\big|^2+2\bigl|\n_{i,j}^m \Delta_h \W_{i,j}^m+\n_{i,j}^m \G(p_{i,j}^m)\bigr|^2\\
&\leq 8\max_{i,j} |\Grad_h \W_{i,j}^m|\bigl\{ | u_{i+1/2,j}^m|\,  |D^+_1 \n_{i,j}^m\big|^2+| u_{i-1/2,j}^m|\,  |D^-_1 \n_{i,j}^m\big|^2\\
&\hphantom{\leq 8\max_{i,j} |\Grad_h \W_{i,j}^m|\bigl\{}+| v_{i,j+1/2}^m|\,  |D^+_2 \n_{i,j}^m\big|^2+| v_{i,j-1/2}^m|\,  |D^-_2 \n_{i,j}^m\big|^2\bigr\}\\
&\hphantom{\leq  }+2\bigl|\n_{i,j}^m \Delta_h \W_{i,j}^m+\n_{i,j}^m \G(p_{i,j}^m)\bigr|^2
\end{align*}
Thus if we assume that $\Dt$ satisfies the CFL-condition \eqref{eq:CFL2}, we have
\begin{equation*}
\begin{split}
{\Dt}\sum_{i,j} |D_t^+ \n_{i,j}^m|^2&\leq h \sum_{i,j} \bigl\{ |u_{i+1/2,j}^m| |D_1^+ \n_{i,j}^m|^2+ |v_{i,j+1/2}^m| |D_2^+\n_{i,j}^m|^2\bigr\} \\
&\quad+ h\sum_{i,j} \big|\n_{i,j}^m\Delta_h \W_{i,j}^m+ \n_{i,j}^m \G(p_{i,j}^m)\big|^2
\end{split}
\end{equation*}
Now summing \eqref{eq:L2ent} over all $i,j$, multiplying with $h^d$ and using the latter inequality, we obtain
\begin{equation*}
\begin{split}
h^d D_t^+\sum_{i,j} f_{i,j}^m&=-h^{d+1} \sum_{i,j}\left(|u_{i-1/2,j}^m| |D^-_1 \n_{i,j}^m|^2+ |v_{i,j-1}^m| |D^-_2 \n_{i,j}^m|^2\right)\\
&\hphantom{= }+h^d\Dt \sum_{i,j} |D_t^+\n_{i,j}^m|^2+ h^d\sum_{i,j} f_{i,j}^m\left(\Delta_h \W_{i,j}^m+2\G(p_{i,j}^m)\right) \\
&\leq h^d\sum_{i,j} f_{i,j}^m\left(\Delta_h \W_{i,j}^m+2\G(p_{i,j}^m)\right)\\
&\hphantom{\leq }+h^{d+1}\sum_{i,j} |\n_{i,j}^m\Delta_h \W_{i,j}^m+\n_{i,j}^m\G(p_{i,j}^m)|^2 \\
&\leq C,
\end{split}
\end{equation*}
where $C>0$ is a constant independent of $h$, thanks to the $L^\infty$-bounds on $\n_h$ and $\Delta_h \W_h$ obtained in Lemma \ref{lem:linfn_h} and  \ref{lem:Wh}. This implies that
\begin{align*}
&h^{d+1}\Dt \sum_{m=0}^{N_T}\sum_{i,j}\left(|u_{i-1/2,j}^m| |D^-_1\n_{i,j}^m|^2+ |v_{i,j-1/2}^m| |D^-_2 \n_{i,j}^m|^2\right)\leq C\\
 & h^d \Dt^2 \sum_{m=0}^{N_T}\sum_{i,j} |D_t^+ \n_{i,j}^m|^2\leq C.
\end{align*}
and therefore using H\"{o}lder's inequality and the uniform $L^\infty$-bounds on $\n_h$, \eqref{eq:dritt}. Using summation by parts, we realize that the other terms, \eqref{eq:t1} -- \eqref{eq:t5} are in $L^\infty([0,T];W^{-1,q}(\dom))$ for $q\in [1,2^*)$ where $2^*=2d/(d-2)$ if $d\geq 3$ and any finite number greater than one if $d=2$.
\end{proof}
\begin{remark}
The preceeding lemma implies that the forward time difference of the approximation of the pressure $D_t^+ p_h=D_t^+|\n_h|^\gamma$ is of the form $D_t^+p_h=g_h+k_h$ where $g_h\in L^1([0,T]\times \dom)$ and $k_h\in L^\infty([0,T]; W^{-1,q}(\dom))$ for any $1\leq q<\infty$ if $d=2$ and for $1\leq q \leq q^*=2d/(d-2)$ if $d>2$, uniformly in $h>0$. Using this, we have that $D_t^+\W_h=U_h+V_h$ where $U_h$ and $V_h$ solve
\begin{equation*}
-\mu\Delta_h U_h +U_h=g_h,\quad\mathrm{and}\quad -\mu\Delta_h V_h +V_h=k_h.
\end{equation*}
By Lemma \ref{lem:ellL1}, we have $U_h,\Grad_h U_h \in L^1([0,T]; L^q(\dom))$ for $1\leq q\leq d/(d-1)$ and by standard results, $V_h$, $\Grad_h V_h\in L^\infty([0,T]; L^2(\dom))$. Hence $D_t \W_h, D_t \Grad_h \W_h\in  L^1([0,T]; L^q(\dom))+ L^\infty([0,T]; L^2(\dom))$.
\end{remark}
\begin{remark}[CFL-condition]
The estimates from Lemma \ref{lem:Wh} imply that the velocity $\vc{u}_h:=\Grad_h\W_h\in L^\infty([0,T];L^{2^*}(\dom))$ uniformly in $h>0$, $2^*=2d/(d-2)$ or any number in $[1,\infty)$ if $d=2$, using the Sobolev embedding theorem. Using an inverse inequality, we can bound it in the $L^\infty(\Dom)$-norm as follows:
\begin{equation*}
\max_{(x,t)\in\Dom}|\vc{u}_h|\leq C h^{-\frac{d}{2^*}}\left(\int_{\dom}|\vc{u}_h|^{2^*} dx\right)^{\frac{1}{2^*}}\leq C h^{-\frac{d}{2^*}}
\end{equation*}
Thus the time step size $\Dt$ is of order $\mathcal{O}(h^{1+d/2^*})$. In practice a linear CFL-condition seems to work well though.
\end{remark}

\subsection{Passing to the limit $h\rightarrow 0$}
The estimates of the previous (sub)sections allow us to pass to the limit $h\rightarrow 0$ in a subsequence still denoted $h$,
\begin{align*}
\n_h\weak \n\geq 0,&\quad \mathrm{in}\, L^q([0,T]\times \dom), \, 1\leq q<\infty,\\
p_h\weak \overline{p}\geq 0,&\quad \mathrm{in}\, L^q([0,T]\times \dom), \, 1\leq q<\infty,
\end{align*} 
where $p_h:=\n_h^\gamma$ and $0\leq \n,\overline{p}\in L^{\infty}([0,T]\times \dom)$. Using the ``discretized'' Aubin-Lions lemma \ref{lem:A-L-lem} for $\W_h$ and $\Grad_h \W_h$, we obtain strong convergence of a subsequence in $L^q([0,T]\times\dom)$ for any $q\in [0,\infty)$ in the case of $\W_h$ and $1 \leq q\leq 2^*$ in the case of $\Grad_h \W_h$ ($2^*=2d/(d-2)$ if $d\geq 3$ and any finite number greater than or equal to one if $d=2$), to limit functions $\W, \Grad\W\in L^q([0,T]\times \dom)$. Moreover, from the estimates in Lemma \ref{lem:Wh} we obtain that $\W\in L^{\infty}([0,T]\times\dom)\cap L^{\infty}([0,T];H^2(\dom))$.
Hence we have that $(\n,\W,\overline{p})$ satisfy for any $\varphi,\psi \in C^1([0,T]\times\dom)$,
\begin{align*}
\lint{\n\varphi_t-\n\Grad\W\cdot \Grad\varphi}&=-\lint{\overline{\n\G({p})}\varphi}\\
\lint{\W\psi+\mu\Grad\W\cdot\Grad\psi}&=\lint{\overline{p}\,\psi}
\end{align*}
where $\overline{\n\G({p})}$ is the weak limit of $\n_h\G(p_h)$. %
To conclude that the limit $(\n,\W,p)$ is a weak solution of \eqref{HeleShaw}, we proceed as in the previous Section \ref{S4} and show that $\n_h$ in fact converges strongly: First, we recall that the limit $\n$ satisfies \eqref{eq:faen1}.
%
%

On the other hand, from \eqref{eq:L2ent}, we obtain (under the CFL-condtion \eqref{eq:CFL2})
\begin{align}\label{eq:schissi}
\begin{split}
D_t^+|\n_{i,j}^m|^2 &\leq \frac{1}{2} D^+_1\left( u_{i-1/2,j}^m  \left(|\n_{i,j}^m|^2+|\n_{i-1,j}^m|^2\right)\right)\\
&\hphantom{=}+ \frac{1}{2}D^+_2\left(v_{i,j-1/2}^m\left( |\n_{i,j}^m|^2+|\n_{i,j-1}^m|^2\right)\right)\\
&\hphantom{=} -\frac{h^2}{2} D^-_1\left[u_{i+1/2,j}^m |D^+_1 \n_{i,j}^m|^2\right]-\frac{h^2}{2}D^-_2\left[ v_{i,j+1/2}^m |D^+_2 \n_{i,j}^m|^2\right]\\
&\hphantom{=}+\frac{h}{2} D^-_1 \left[\n_{i,j}^m|u_{i+1/2,j}| D^+_1 \n_{i,j}^m\right]+\frac{h}{2} D^-_2 \left[\n_{i,j}^m|v_{i,j+1/2}| D^+_2 \n_{i,j}^m\right]  \\
&\hphantom{=}+\frac{h}{2} D^+_1 \left[\n_{i,j}^m|u_{i-1/2,j}| D^-_1 \n_{i,j}^m\right] +\frac{h}{2} D^+_2 \left[\n_{i,j}^m|v_{i,j-1/2}| D^-_2 \n_{i,j}^m\right]\\
&\hphantom{=}+|\n_{i,j}^m|^2 \Delta_h \W_{i,j}^m+2|\n_{i,j}^m|^2 \G(p_{i,j}^m),
\end{split}
\end{align}
Considering this inequality in terms of the piecewise constant functions $\n_h$, $\W_h$ and $p_h$, multiplying it with a nonnegative $C^1$-test function $\varphi$, integrating and then passing to the limit $h\rightarrow 0$, we obtain (using the bounds \eqref{eq:dritt}, the weak convergence of $\n_h$ and $p_h$ and the strong convergence of $\W_h$ and $\Grad_h\W_h$),
\begin{equation}\label{eq:faen2}
-\lint{\overline{\n^2} \varphi_t- \overline{\n^2}\Grad \W\cdot\Grad \varphi} \leq \lint{\left(\overline{\n^2\Delta \W}+ 2 \overline{\n^2 \G(p)}\right)\varphi},
\end{equation}
where $\overline{\n^2}$ denotes the weak limit of $\n_h^2$ and $\overline{\n^2\Delta \W}$ and $\overline{\n^2 \G(p)}$ are the weak limits of $\n_h^2\Delta_h \W_h$ and $\n_h^2 \G(p_h)$ respectively.

Adding \eqref{eq:faen1} and \eqref{eq:faen2}, we have
\begin{multline*}
-\lint{\left(\overline{\n^2}-\n^2 \right)\varphi_t-\left(\overline{\n^2} -\n^2\right)\Grad\W\cdot \Grad\varphi}\\
\leq \lint{\left(2 \overline{\n^2 \G(p)}-2n\overline{\n\G({p})}+\overline{\n^2\Delta \W}-n^2\Delta\W\right)\varphi}.
\end{multline*}
We now choose smooth test functions $\varphi_\eps$ approximating $\varphi(t,x)=\mathbf{1}_{[0,\tau]}(t)$, where $\tau\in (0,T]$, in this inequality and then pass to the limit $\eps\rightarrow 0$ to obtain
\begin{multline}\label{eq:schafseckel}
\int_{\dom} \left(\overline{\n^2}-\n^2\right)\!(\tau) dx-\int_{\dom} \left(\overline{\n^2}(0,x)-\n^2(0,x)\right)\! dx\\
\leq \int_0^\tau\!\!\int_{\dom}\!\left(2 \overline{\n^2 \G(p)}-2n\overline{\n\G({p})}+\Delta\W\left(\overline{\n^2}-\n^2\right)+\overline{\n^2\Delta \W}-\overline{n^2}\Delta\W\right)\! dx\, dt.
\end{multline}
By convexity of $f(x)=x^2$, we have $\overline{\n^2}\geq \n^2$, on the other hand, the discrete $L^2$-entropy inequality, \eqref{eq:schissi}, implies
\begin{equation*}
\int_{\dom} |\n_h(\tau,x)|^2\, dx\leq \int_{\dom} |\n_h^0|^2 dx+\int_0^\tau\!\!\int_{\dom}\left(|\n_h|^2\Delta_h\W_h+ 2|\n_h|^2\G(p_h)\right)dx dt,
\end{equation*}
which gives, passing to the limit $h\rightarrow 0$,
\begin{equation*}
\int_{\dom} \overline{|\n|^2}(\tau,x)\, dx\leq \int_{\dom} |\n_0|^2dx+\int_0^\tau\!\!\int_{\dom}\left(\overline{|\n|^2\Delta\W}+ 2\overline{|\n|^2\G(p)}\right)dx dt.
\end{equation*}
Letting $\tau\rightarrow 0$, the second term on the right hand side vanishes (as the integrand is bounded), and we obtain
\begin{equation*}
\int_{\dom} \overline{|\n|^2}(0,x)\, dx\leq \int_{\dom} |\n_0|^2 dx
\end{equation*}
We deduce that $\overline{|\n|^2}(0,\cdot)= |\n_0|^2$ almost everywhere and that therefore the second term on the left hand side of \eqref{eq:schafseckel} is zero.
%
We have already estimated the first two terms on the right hand side of \eqref{eq:schafseckel} in \eqref{eq:estimate1} and \eqref{eq:helvete}. To bound the other term, we use a discretized version of Lemma \ref{lem:effve}:
%
%
%
%
\begin{lemma}\label{lem:effv}
The weak limits $(\n,\W,\weakl{p})$ of the sequences $\{(\n_h,\W_h,p_h)\}_{h>0}$ satisfy for any smooth function $S:\R\rightarrow\R$,
\begin{equation}\label{eq:gugus}
\int_{\dom}\left(\weakl{S(\n)\Delta\W}- \weakl{S(\n)}\Delta\W\right) dx = \frac{1}{\mu} \int_{\dom}\left(\weakp{p}{S(\n)}- \weakl{pS(\n)}\right)dx
\end{equation}
where $\overline{S(\n)\Delta\W}$, $\overline{S(\n)}$, $\overline{pS(\n)}$ are the weak limits of $S(\n_h)\Delta_h\W_h$, $S(\n_h)$ and $p_hS(\n_h)$ respectively.
\end{lemma}
Applying this lemma to the last term in \eqref{eq:schafseckel2} with $S(\n)=\n^2$, we can estimate it by
\begin{align*}
\int_{0}^\tau\int_{\dom}\left(\weakl{\n^2\Delta\W}- \weakl{\n^2}\Delta\W\right) dx &= \frac{1}{\mu} \int_{\dom}\left(\weakp{p}{\n^2}- \weakl{p \n^2}\right)dxdt\\
&= \frac{1}{\mu} \int_{\dom}\left(\weakp{\n^{\gamma}}{\n^2}- \weakl{ \n^{2+\gamma}}\right)dxdt\\
&\leq 0,
\end{align*}
using again that by Exercise 3.37 in \cite{Novotny-Stras-2004}, $\weakp{\n^{\gamma}}{\n^2}\leq \weakl{ \n^{2+\gamma}}$. Thus,
\begin{equation*}
\int_{\dom} \left(\weakl{\n^2}-\n^2\right)\!(\tau) dx \leq \left(2\alpha+\frac{P_M}{\mu}\right) \int_0^\tau\!\!\int_{\dom}\left(\weakl{\n^2}-\n^2\right) dx\, dt.
\end{equation*}
Gr\"{o}nwall's inequality thus implies
\begin{equation*}
\int_{\dom} \left(\overline{\n^2}-\n^2\right)\!(\tau) dx\leq 0
\end{equation*}
By convexity of the function $f(x)=x^2$ we also have $\n^2\leq \overline{\n^2}$ almost everywhere and hence
\begin{equation*}
\overline{\n^2}=\n^2
\end{equation*}
almost everywhere in $\Dom$. Therefore we conclude that the functions $\n_h$ converge strongly to $\n$ almost everywhere, thus also $\overline{p}=\n^\gamma$ and so the limit $(\n,\W,\overline{p})$ is a weak solution of the equations \eqref{HeleShaw}.
\begin{proof}[Proof of Lemma \ref{lem:effv}]
We multiply the equation for $\W_h$ by $S(\n_h)$ and integrate it over the spatial domain $\dom$,
  \begin{equation*}
  \int_{\dom}\mu\Delta_h\W_h\,S(\n_h)-\W_hS(\n_h)\,dx = -\int_{\dom} p_h S(\n_h) \,dx. 
  \end{equation*}
Passing to the limit $h\rightarrow 0$ in the last equation, we obtain
\begin{equation}\label{eq:dubel}
 \int_{\dom}\mu\overline{\Delta\W S(\n)}-\W \overline{S(\n)}\,dx =-\int_{\dom} \overline{ p S(\n)} \,dx.
\end{equation}
On the other hand, using $[S(\n_h)\ast \psi_\delta](x)$, where $\psi_\delta$ is a smooth mollifier converging to a Dirac measure at zero when $\delta$ is sent to zero, as a test function in the weak formulation of the limit equation
\begin{equation*}
-\mu\Delta\W+\W= \overline{p},
\end{equation*}
and passing first to the limit $\delta\rightarrow 0$ and then $h\rightarrow 0$, we obtain
\begin{equation*}
\int_{\dom} \mu\Delta\W \overline{S(\n)} -\W\overline{S(\n)}\, dx=-\int_{\dom}  \weakp{p}{S(\n)}\, dx
\end{equation*}
Combining the last identity with \eqref{eq:dubel}, we obtain \eqref{eq:gugus}.
\end{proof}
\section{Numerical examples}
\label{S5}
To test the scheme in practice, we compute approximations for the following two examples.
\subsection{Gaussian initial data}
As a first example, we consider the initial data
\begin{equation}\label{eq:init1}
\n_0(x)=\frac{1}{2}\exp\left(-10\left(x_1^2+x_2^2\right)\right),
\end{equation}
on the domain $\dom=[-2.5,2.5]^2$ and $h=1/64$ with pressure law $p=\n^3$ and $\G(p)=1-p$ and $\mu=1$.
Strictly speaking, these are not homogeneous Neumann boundary conditions, but since the gradient of $\n_0$ near the boundary is very small, this works well in practice. 
\begin{figure}[ht]
  \centering
  \begin{tabular}{lr}
    \includegraphics[width=0.5\textwidth]{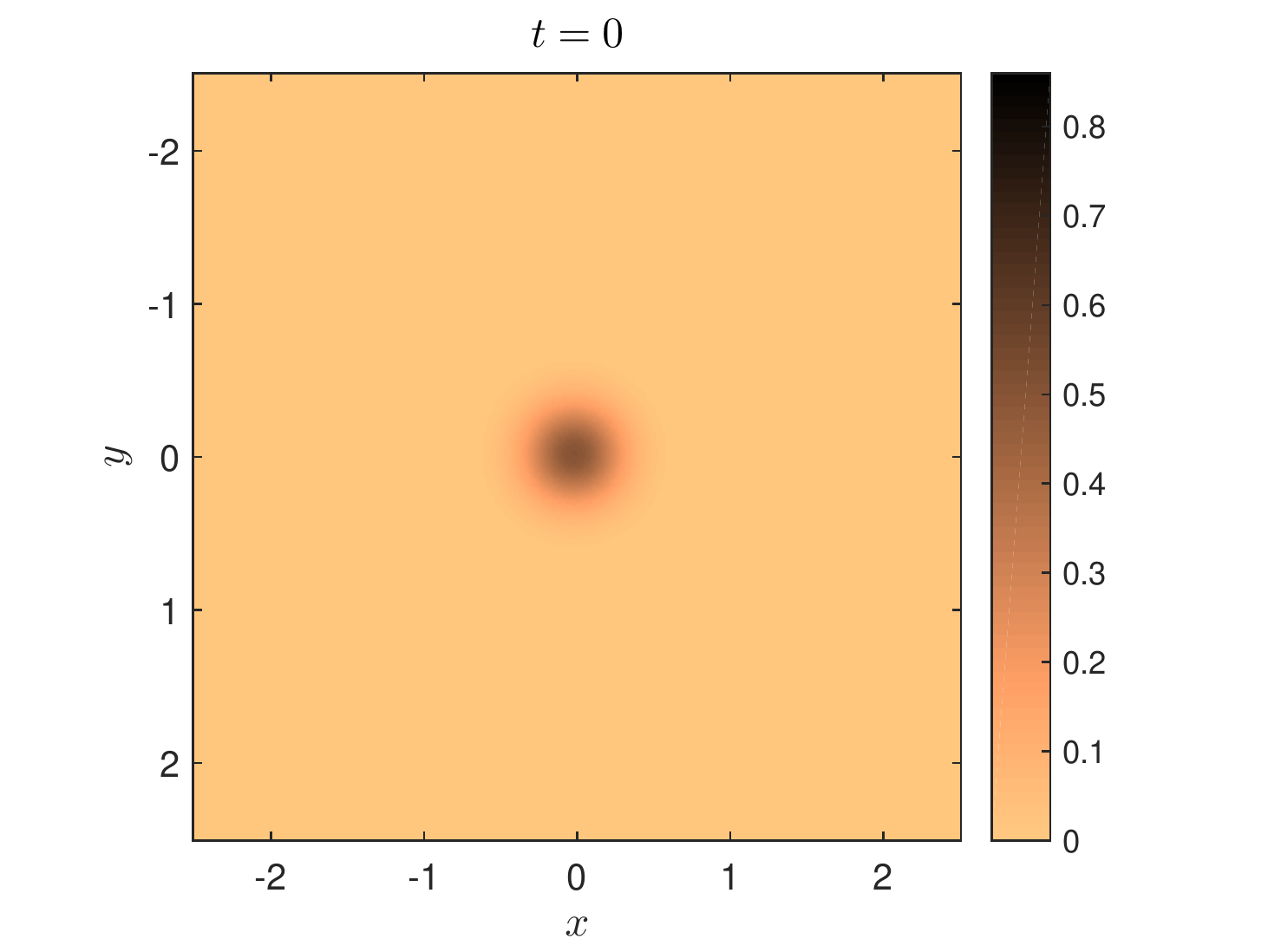}
    \includegraphics[width=0.5\textwidth]{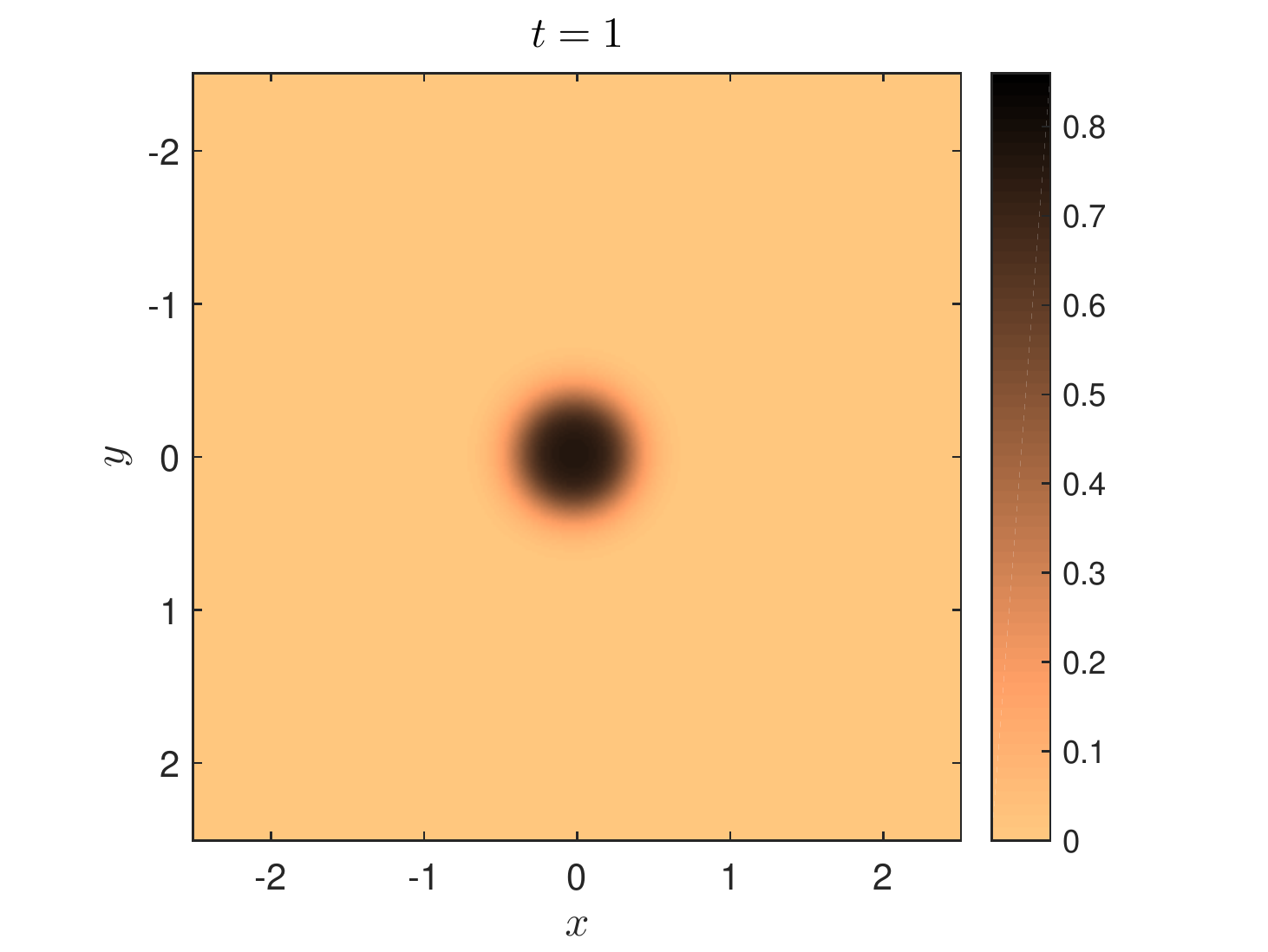}\\
    \includegraphics[width=0.5\textwidth]{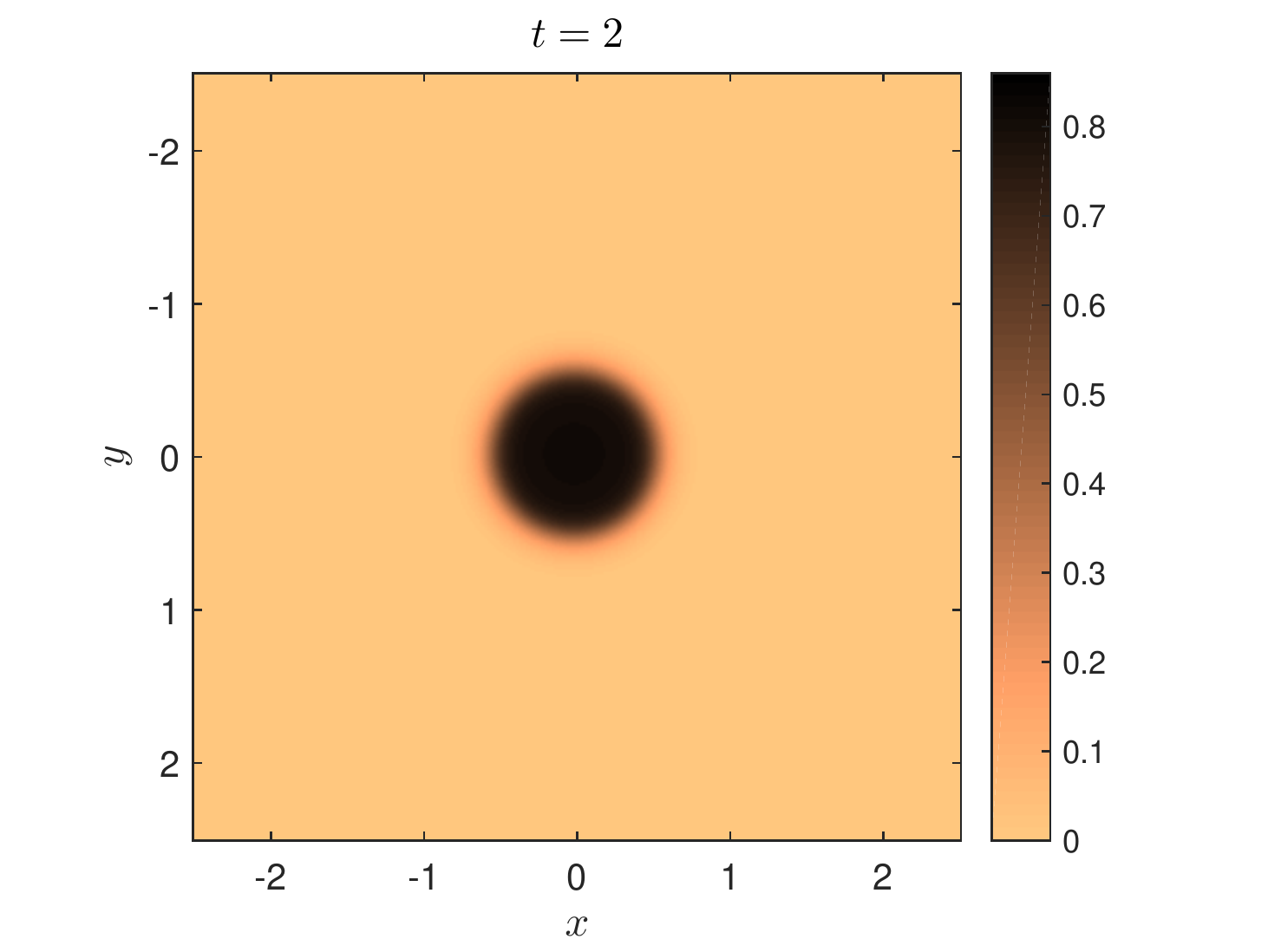}
    \includegraphics[width=0.5\textwidth]{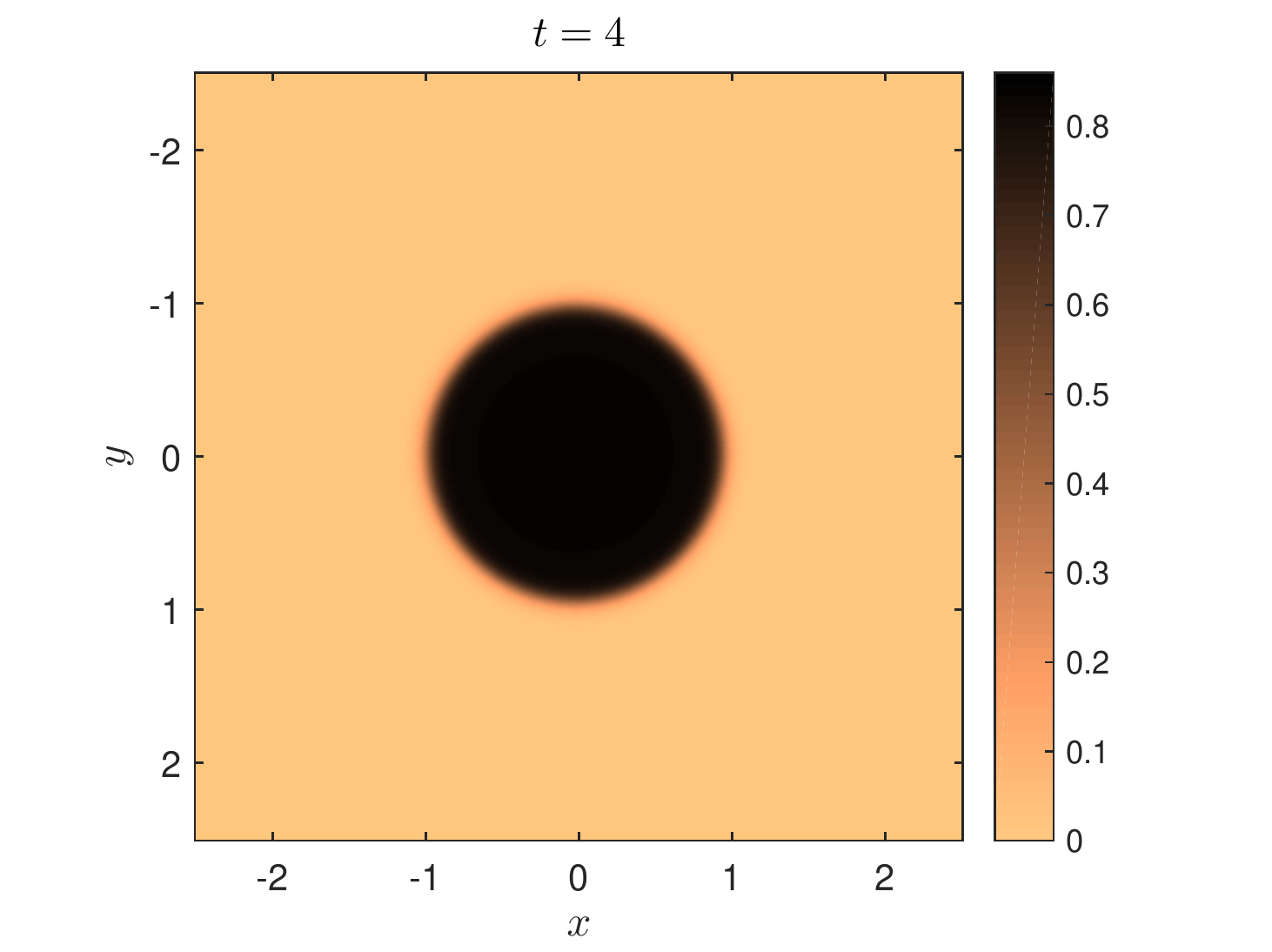}
  \end{tabular}
  \caption{The approximations of the cell density $\n$ for initial data \eqref{eq:init1} on $\dom=[-2.5,2.5]^2$ with mesh width $h=1/64$.}
  \label{fig:eg1}
\end{figure}

In Figure \ref{fig:eg1} we show the approximations at times $t=0,1,2,4$. We observe that the cell density in the middle first reaches the maximum possible and then starts spreading with a relatively narrow transition region between zero density and maximum density. 
\subsection{Two Gaussians}
As a second example, we use the inital data consisting of two Gaussian pulses with centers at $x=(0.7,0)$ and $x=(-0.6,0.2)$,
\begin{equation}\label{eq:init2}
\begin{split}
\n_0(x)&=\frac{1}{2}\exp\left(-10\left((x_1-0.7)^2+x_2^2\right)\right)\\
&\quad+\frac{1}{2}\exp\left(-20\left((x_1+0.6)^2+(x_2-0.2)^2\right)\right)
\end{split}
\end{equation}
on the same domain, $\dom=[-2.5,2.5]^2$, with $\mu=1$, pressure law $p=\n^{10}$ and $\G(p)=1-p$ and mesh width $h=1/64$.
\begin{figure}[ht]
  \centering
  \begin{tabular}{lr}
    \includegraphics[width=0.5\textwidth]{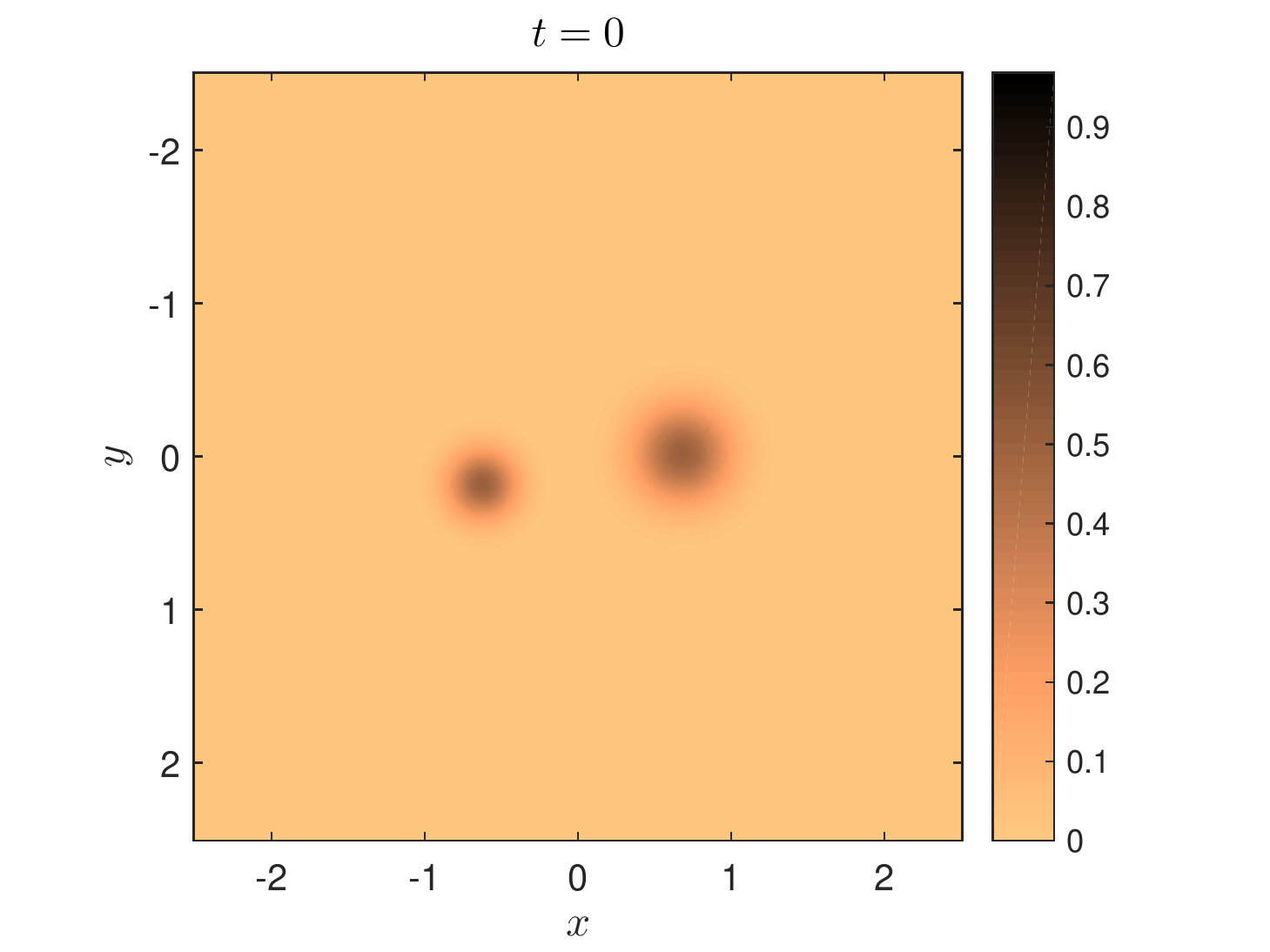}
    \includegraphics[width=0.5\textwidth]{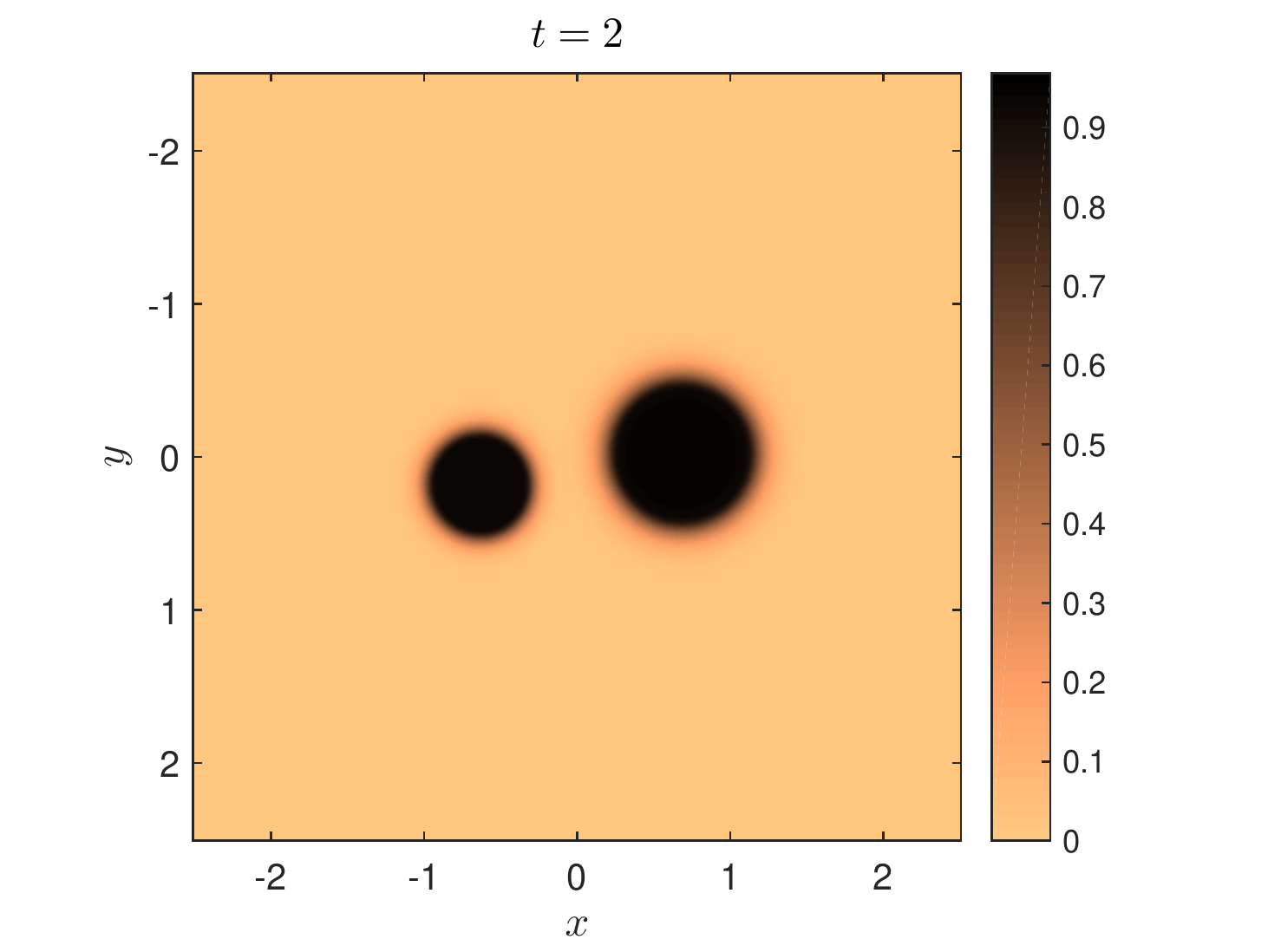}\\
    \includegraphics[width=0.5\textwidth]{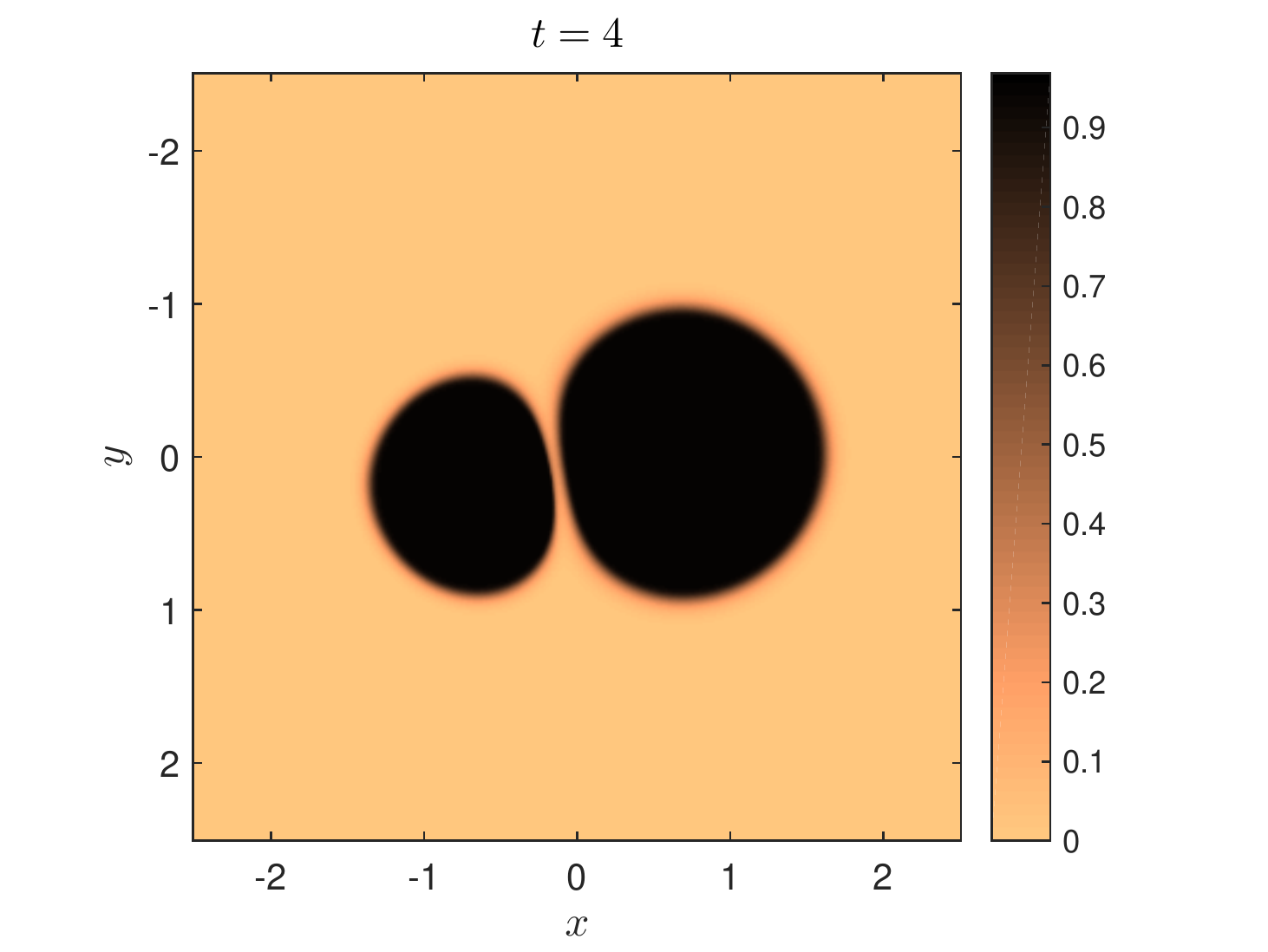}
    \includegraphics[width=0.5\textwidth]{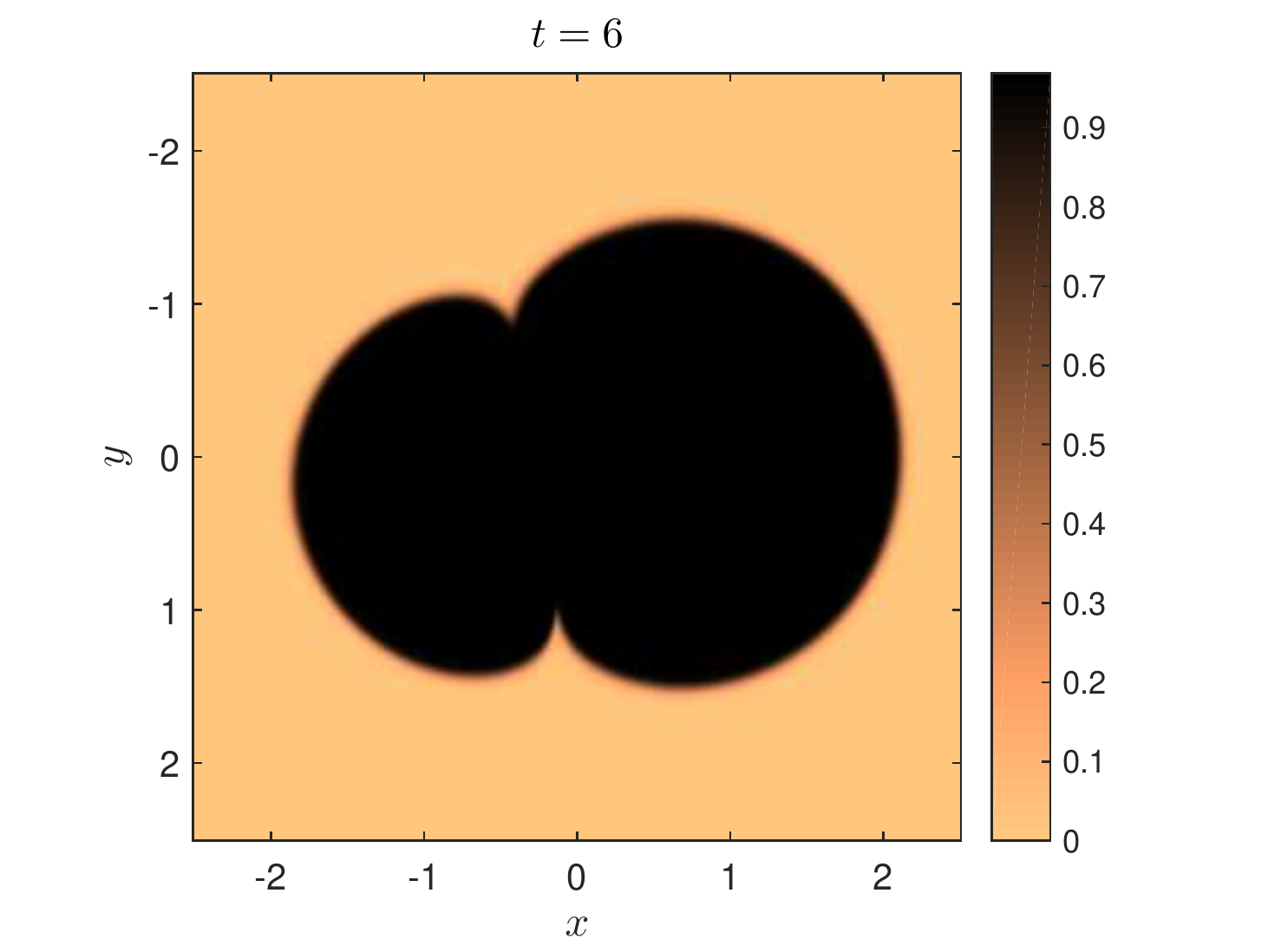}
  \end{tabular}
  \caption{The approximations of the cell density $\n$ for initial data \eqref{eq:init2} on $\dom=[-2.5,2.5]^2$ with mesh width $h=1/64$.}
  \label{fig:eg2}
\end{figure}
The approximations computed at times $t=0,2,4,6$ are shown in Figure \ref{fig:eg2}. The interface between the area with maximum cell density and zero cell density seems to be sharper than in the previous example, this appears to be caused by the pressure law with the higher exponent $\gamma$. Further tests with higher and lower exponents confirmed that assertion.

\appendix
\section{Discretized Aubin-Lions lemma}
\begin{lemma}\label{lem:A-L-lem}
Let $u_h:\dom\rightarrow \R^k$ be a piecewise constant function defined on a grid on $[0,T)\times \dom$, $\dom$ a bounded rectangular domain, satisfying
\begin{equation}\label{eq:al-ass1}
\lint{ |u_h|^q+|\Grad_h u_h|^q}\leq C
\end{equation}
for some $\infty>q>1$, uniformly with respect to $h>0$ and 
\begin{equation}\label{eq:al-ass2}
D_t u_h= A_h  f_h+  g_h + k_h,
\end{equation}
where $A_h$ is a first order linear finite difference operator, and $f_h, g_h, k_h:\dom\rightarrow \R^{d\times k}$ are piecewise constant functions, satisfying uniformly in $h>0$,

\begin{equation}\label{eq:al-ass3}
\lint{ |f_h|^{r_1}+|g_h|^{r_2}+|k_h|}\leq C,
\end{equation}
for some $\infty>r_1,r_2>1$. Then $u_h\rightarrow u$ in $L^q([0,T)\times \dom)$.
\end{lemma}
\begin{proof}
Denote $\wu_h$ a piecewise linear interpolation of $u_h$ in space piecewise constant in time and similarly, let $\widehat{g}_h$, $\widehat{f}_h$ and $\widehat{k}_h$ piecewise linear interpolations of $g_h$, $f_h$ and $k_h$ respectively in space and piecewise constant in time such that
\begin{equation}\label{eq:jumalauta}
D_t \wu_h=A_h \widehat{f}_h + \widehat{g}_h+ \widehat{k}_h.
\end{equation}
By Ladyshenskaya's norm equivalences \cite[p. 230 ff]{Lady1985}, we have
\begin{equation*}
\begin{split}
\int_0^T \|\wu_h\|_{W^{1,q}(\dom)}^q\, dt&\leq C \lint{|u_h|^q+|\Grad_h u_h|^q}\\
\int_0^T\! \|\widehat{f}_h\|_{L^{r_1}(\dom)}^{r_1}\!+\!\|\widehat{g}_h\|_{L^{r_2}(\dom)}^{r_2}\!+\!\|\widehat{k}_h\|_{L^1(\dom)} dt &\leq C\!\lint{|f_h|^{r_1} + |g_h|^{r_2}+|k_h|}
\end{split}
\end{equation*}
where the right hand sides are bounded by assumptions \eqref{eq:al-ass1} and \eqref{eq:al-ass3}.
Since $L^1(\dom)\subset W^{-1,s}(\dom)$ for $1\leq s\leq 1^*=d/(d-1)$, we have that $\widehat{k}_h\in L^1([0,T];W^{-1,s}(\dom))$ for $1\leq s\leq 1^*=d/(d-1)$ and hence thanks to this and \eqref{eq:jumalauta}, we obtain
\begin{equation*}
\wu_h\in L^q([0,T);W^{1,q}(\dom)),\quad D_t\wu_h\in L^1([0,T); W^{-1,\min\{r_1,1^*\}}(\dom)),
\end{equation*}
uniformly with respect to the discretization parameter $h>0$. Thus we can apply the version \cite[Theorem 1]{Dreher2012} of the Aubin-Lions lemma to find that up to a subsequence $\wu_h\rightarrow u$ in $L^q([0,T)\times \dom)$ and the limit $u\in L^q([0,T); W^{1,q}(\dom))$. By \cite[Lemma 3.2., p. 226]{Lady1985} this implies that also $u_h\rightarrow u$ in $L^q([0,T)\times \dom)$ (and $\Grad_h u_h\weak \Grad u$ ).
%
%
%
\end{proof}
\begin{remark}[Derivatives]
If the $u_h$ in Lemma \ref{lem:A-L-lem} are of the form $\Grad_h v_h$ for some $v_h$ piecewise constant function, this lemma implies that $\Grad_h v_h\rightarrow \Grad v$ in $L^q$, again applying \cite[Lemma 3.2., p. 226]{Lady1985}
\end{remark}
\section{Technical Lemmas}
In this section, we prove the following lemma:
\begin{lemma}\label{lem:ellL1}
Let $u_h$ solve the difference equation
\begin{equation}\label{eq:do}
-\Div_h(A_h \Grad_h u_h)+c_h u_h=f_h,\quad x\in\dom,
\end{equation}
with homogeneous Neumann boundary conditions, where $A_h$ is a diagonal positive definite $d\times d$-matrix with entries $a_h^{(ii)}\geq \eta>0$ and $c_h\geq\nu>0$ uniformly in $h>0$, $x\in\dom$, $\dom$ is a rectangular domain in $\R^d$ and 
$$\|f_h\|_{L^1(\dom)}\leq M,$$
uniformly in $h>0$. We have denoted $\Grad_h:=\Grad_h^-$ and $\Div_h:=\Div_h^+$ (or alternatively $\Grad_h:=\Grad_h^+$ and $\Div_h:=\Div_h^-$). Then
\begin{equation*}
\|u_h\|_{L^q(\dom)}+\|\Grad_h u_h\|_{L^q(\dom)}\leq C,
\end{equation*}
where $1\leq q<d/(d-1)$, for a constant $C>0$ independent of $h>0$.
\end{lemma}
The proof of this lemma will be a (simplified) finite difference version of the proof of Theorem 2.1 in \cite{l1rhs}.
But before proving the lemma, we need to introduce some notation.
\begin{notation}\label{not:marcinkiewicz}
For any $r\in (1,\infty)$, we denote by $L^{r,\infty}(\dom)$ the Marcinkiewicz space with norm defined by 
\begin{equation*}
\|u\|_{L^{r,\infty}(\dom)}=\sup_{\lambda>0} \lambda |\{x\in\dom: |u(x)\geq \lambda\}|^{1/r}.
\end{equation*}
\end{notation}
The Marcinkiewicz spaces are continuously embedded in $L^q(\dom)$ for any $1\leq q<r$, \cite{gilbargtrudinger}:
\begin{equation}\label{eq:soad}
\|u\|_{L^q(\dom)}\leq C(q,r,|\dom|) \|u\|_{L^{r,\infty}(\dom)},\quad q\in[1,r).
\end{equation}
Moreover, we need the trunctation operator $S_k$ defined as follows:
\begin{notation}\label{not:truncation}
Let $k>0$ be a real number. Then we define the truncation operator $S_k:\R\rightarrow \R$ by
\begin{equation*}
S_k(s)=\begin{cases}
 s,& \textrm{if}\,\, |s|\leq k,\\
 k \frac{s}{|s|},&\textrm{if}\,\, |s|\geq k.
\end{cases}
\end{equation*}
\end{notation}
It will be convenient in the proof to use the following tuple notation for the finite difference approximations:
\begin{notation}\label{not:tuple}
We denote $\underline{i}:=(i_1,\dots,i_d)$, $i_\ell=1,\dots, N_\ell$, $N_\ell$ the number of cells in the $\ell$th spatial direction, a $d$-dimensional tuple and and $u_{\underline{i}}$ the approximation in cell $\cell_{\underline{i}}:=((i_1-1)h,i_1h]\times\dots\times ((i_d-1)h,i_d h]$. The piecewise constant function $u_h$ can be written as
\begin{equation*}
u_h(x):=\sum_{\underline{i}} u_{\underline{i}}\, \mathbf{1}_{\cell_{\underline{i}}}(x),\quad x\in \dom.
\end{equation*}
\end{notation}
We also need the following auxilary result:
\begin{lemma}\label{lem:truncL2}
Let $u_h$ solve the difference equation \eqref{eq:do} under the assumptions of Lemma \ref{lem:ellL1}. Then
\begin{equation}\label{eq:gaga}
\int_{\dom} |\Grad_h S_k(u_h)|^2+|S_k(u_h)|^2 dx\leq C M k,\quad \forall \, k>0,
\end{equation}
for some constant $C>0$ independent of $h>0$.
\end{lemma}
\begin{proof}
Given $k>0$, we multiply equation \eqref{eq:do} by $S_k(u_h)$ and integrate over the domain $\dom$. After changing variables in the integrals, we obtain
\begin{equation}\label{eq:dubeli}
\int_{\dom} \left(A_h \Grad_h u_h\right)\cdot \Grad_h S_k(u_h) +c_h u_h S_k(u_h)\, dx=\int_{\dom} f_h S_k(u_h)\, dx.
\end{equation}
The right hand side can be bounded by $M k$ using H\"{o}lder's inequality. The left hand side, we can rewrite and estimate as follows
\begin{equation*}
\begin{split}
&\int_{\dom} \left(A_h \Grad_h u_h\right)\cdot \Grad_h S_k(u_h) +c_h u_h S_k(u_h)\, dx\\
&\quad = \int_{\dom} \left(A_h \Grad_h S_k(u_h)\right)\cdot \Grad_h S_k(u_h) +c_h  |S_k(u_h)|^2\, dx\\
&\quad\quad+     \int_{\dom} \left(A_h \left(\Grad_h \left[u_h-S_k(u_h)\right]\right)\right)\cdot \Grad_h S_k(u_h) +c_h \left(u_h-S_k(u_h)\right) S_k(u_h)\, dx\\
&\quad \geq \eta \|\Grad_h S_k(u_h)\|^2_{L^2(\dom)}+\nu \|S_k(u_h)\|^2_{L^2(\dom)} \\
&\quad\quad+ \int_{\dom} \left(A_h \left(\Grad_h \left[u_h-S_k(u_h)\right]\right)\right)\cdot \Grad_h S_k(u_h) +c_h \left(u_h-S_k(u_h)\right) S_k(u_h)\, dx.
\end{split}
\end{equation*}
$(u_h-S_k(u_h))$ is either zero or has the same sign as $S_k(u_h)$. Therefore $(u_h-S_k(u_h)) S_k(u_h)\geq 0$ and 
\begin{equation*}
 \int_{\dom}c_h \left(u_h-S_k(u_h)\right) S_k(u_h)\, dx\geq 0.
\end{equation*}
In order to prove that the other term is positive as well, we will show that 
\begin{equation*}
D^-_\ell S_k(u_{\underline{i}}) D^-_\ell \left(u_{\underline{i}}-S_k(u_{\underline{i}})\right)\geq 0,\quad \forall \,\underline{i},\,\,\ell=1,\dots, d.
\end{equation*}
The proof of this fact consists of boring case distinctions and is exactly analoguous for $\ell=1,2, (3)$, therefore we will do it only for $\ell=1$ and omit writing the tuple index $\underline{i}$. Then we have
\begin{equation*}
 D^-_1(u_i-S_k(u_i)) D^-_1 S_k(u_i)=\begin{cases}
(u_i-k)(k-u_{i-1}),& u_i>k,\, |u_{i-1}|\leq k,\\
(u_i+k)(-k-u_{i-1}),& u_i<-k,\, |u_{i-1}|\leq k,\\
0,& |u_i|\leq k,\, |u_{i-1}|\leq k,\\
(-u_{i-1}+k)(u_i-k),& |u_i|\leq k,\, u_{i-1}>k,\\
(-u_{i-1}-k)(u_i+k),& |u_i|\leq k,\, u_{i-1}<-k,\\
0,& u_i>k,\, u_{i-1}>k,\\
0,& u_i<-k,\, u_{i-1}<-k,\\
(u_i-u_{i-1}-2k) 2k,&  u_i>k,\,  u_{i-1}<-k,\\
-(u_i-u_{i-1}+2k)  2k,&  u_i<-k,\, u_{i-1}>k.
\end{cases}
\end{equation*}
The potential reader is welcome to check that these are all the possible cases and that each of the terms on the right hand side is nonnegative.
Thus we have that
\begin{multline*}
\int_{\dom} \left(A_h \Grad_h u_h\right)\cdot \Grad_h S_k(u_h) +c_h u_h S_k(u_h)\, dx\\
 \geq \eta \|\Grad_h S_k(u_h)\|^2_{L^2(\dom)}+\nu \|S_k(u_h)\|^2_{L^2(\dom)} 
\end{multline*}
which implies \eqref{eq:gaga} together with the estimate on the right hand side of \eqref{eq:dubeli}
\end{proof}
\begin{proof}[Proof of Lemma \ref{lem:ellL1}]
First, we note that by the discrete Gagliardo-Nirenberg-Sobolev inequality, \cite[Thm. 3.4]{Bessemoulin2014},
\begin{equation*}
\int_{\dom} |S_k(u_h)|^{2^*} dx\leq C^{2^*}\left(\int_{\dom}|\Grad_h S_k(u_h)|^2+|S_k(u_h)|^2dx\right)^{\frac{2^*}{2}},
\end{equation*}
where $2^*=2d/(d-2)$ if $d\geq 3$ and any number with $1\leq 2^*<\infty$ if $d=2$, and where $C$ is a constant depending on $|\dom|$ but not on $h>0$. By Lemma \ref{lem:truncL2}, we can bound the right hand side and obtain therefore
\begin{equation}\label{eq:kacke}
\int_{\dom} |S_k(u_h)|^{2^*} dx\leq C (k M)^{\frac{2^*}{2}}.
\end{equation}
Now we define the set $\mathcal{B}(k)$ by
\begin{equation*}
\mathcal{B}(k)=\{ \cell_{\underline{i}}\subset \dom : |u_{\underline{i}}|\geq k\}.
\end{equation*}
We have
\begin{equation*}
\int_{\mathcal{B}(k)} |S_k(u_h)|^{2^*}dx\geq k^{2^*} |\mathcal{B}(k)|,
\end{equation*}
and therefore, using \eqref{eq:kacke},
\begin{equation}\label{eq:hmm}
|\mathcal{B}(k)|\leq \frac{1}{k^{2^*}}\int_{\mathcal{B}(k)} |S_k(u_h)|^{2^*}dx\leq \frac{1}{k^{2^*}}\int_{\dom} |S_k(u_h)|^{2^*}dx\leq \frac{C M^{\frac{2^*}{2}}}{k^{\frac{2^*}{2}}}
\end{equation}
which implies that $u_h\in L^{r,\infty}(\dom)$ for $r=2^*/2$ (which is $d/(d-2)$ if $d\geq 3$) since the choice of $k>0$ was arbitrary. Now denote
\begin{align*}
\partial \mathcal{B}(k)&:=\{ \cell_{\underline{i}}\subset \dom : \exists\, \underline{j}, |\underline{i}-\underline{j}|=1,\, |u_{\underline{j}}|\geq k \}\\
\overline{\mathcal{B}(k)}&:=\mathcal{B}(k)\cup \partial \mathcal{B}(k),\\
\mathcal{B}(k)^c&:=\dom\backslash \overline{B}(k),
\end{align*}
where $|\underline{i}-\underline{j}|=\max_{1\leq \ell\leq d} |i_\ell - j_\ell|$. Informally speaking, the cells in $\partial \mathcal{B}(k)$ have a neighbor cell which is contained in $\mathcal{B}(k)$. We have
\begin{equation*}
|\partial \mathcal{B}(k)|\leq (3^d-1) |\mathcal{B}(k)|\leq \frac{C M^{\frac{2^*}{2}}}{k^{\frac{2^*}{2}}},
\end{equation*}
by \eqref{eq:hmm}. Now let $\lambda>0$, $k>0$ and decompose 
\begin{multline*}
\{x\in \dom : |\Grad_h u_h(x)|\geq \lambda\}=\{x\in \dom : |\Grad_h u_h(x)|\geq \lambda\, \mathrm{and}\, x\in \overline{\mathcal{B}(k)}\}\\
\cup \{x\in \dom : |\Grad_h u_h(x)|\geq \lambda\, \mathrm{and}\, x\in \mathcal{B}(k)^c\}.
\end{multline*}
Hence
\begin{equation*}
|\{x\in \dom : |\Grad_h u_h(x)|\geq \lambda\}|\leq |\overline{\mathcal{B}(k)}| + |\{x\in \dom : |\Grad_h u_h(x)|\geq \lambda\, \mathrm{and}\, x\in \mathcal{B}(k)^c\}|.
\end{equation*}
On $\mathcal{B}(k)^c$ and the cells bordering the set, we have $|u_h|\leq k$ and therefore $u_h=|S_k(u_h)|$. Hence we can estimate the size of the second set in the above inequality,
\begin{align*}
&|\{x\in \dom : |\Grad_h u_h(x)|\geq \lambda\, \mathrm{and}\, x\in \mathcal{B}(k)^c\}|\\
&\quad = |\{x\in \dom : |\Grad_h S_k(u_h)(x)|\geq \lambda\, \mathrm{and}\, x\in \mathcal{B}(k)^c\}|\\
&\quad \leq |\{x\in \dom : |\Grad_h S_k(u_h)(x)|\geq \lambda\, \}|\\
&\quad \leq \frac{1}{\lambda^2} \int_{\dom} |\Grad_h S_k(u_h)|^2 dx,
\end{align*}
where we have used Chebyshev inequality for the last step. Now we can estimate the size of the set $\{x\in \dom : |\Grad_h u_h(x)|\geq \lambda\}$ using \eqref{eq:gaga} once more, 
\begin{equation*}
|\{x\in \dom : |\Grad_h u_h(x)|\geq \lambda\}|\leq \frac{C M^{\frac{2^*}{2}}}{k^{\frac{2^*}{2}}}+\frac{C k M}{\lambda^2}.
\end{equation*}
Choosing $k=\lambda^{\frac{4}{2^*+2}}$, we obtain
\begin{equation*}
\lambda^{\frac{2 2^*}{2^*+2}} |\{x\in \dom : |\Grad_h u_h(x)|\geq \lambda\}|\leq C(d,M,|\dom|).
\end{equation*}
 If $d\geq 3$, we have $\frac{2 2^*}{2^*+2}=\frac{d}{d-1}$ and so $u_h, \Grad_h u_h\in L^{r,\infty}(\dom)$ for $1\leq r\leq d/(d-1)$. For $d=2$, since $2^*$ is an arbitrary finite positive number, we can achieve the same. Using the embedding of the Marcinkiewicz spaces, \eqref{eq:soad}, we obtain the claim of the lemma.
\end{proof}

\section*{Acknowlegments}
The work of K.T.  was supported  in part by the National Science Foundation under the grant DMS-1211519. The work of F.W. was supported by the Research Council of Norway, project 214495 LIQCRY.
F.W. gratefully acknowledges the support by the Center for Scientific Computation and Mathematical Modeling  at the University of Maryland where part of this research was performed during her visit in Fall 2014.

\def\cprime{$'$}

\end{document}